\documentclass[12pt,paper=a4]{scrartcl}
\pdfoutput=1

\usepackage[T1]{fontenc}

\usepackage[style=numams]{biblatex}
\usepackage{%
   amssymb,
   mathtools,
   leftidx,
   tikz-cd
}

\renewcommand{\phi}{\varphi}
\renewcommand{\theta}{\vartheta}

\renewcommand{\geq}{\geqslant}
\renewcommand{\leq}{\leqslant}

\newcommand{\ints}{\mathbb{Z}}
\newcommand{\rats}{\mathbb{Q}}
\newcommand{\compl}{\mathbb{C}}

\newcommand{\crp}[1]{\mathbb{#1}}     
\newcommand{\GF}[1]{\mathbb{F}_{#1}}  

\newcommand{\eps}{\varepsilon}

\newcommand*{\cA}{\mathcal{A}}   

\newcommand{\iso}{\cong}    
\newcommand{\nteq}{\mathrel{\trianglelefteq}} 

\newcommand{\cconj}[1]{\overline{#1}}  

\newcommand*{\dcup}{\mathbin{\mathaccent\cdot{\cup}}}   
\DeclareMathOperator*{\bigdcup}{\mathaccent\cdot{\mathop{\bigcup}}}

\DeclarePairedDelimiter{\abs}{\lvert}{\rvert}
\DeclarePairedDelimiter{\erz}{\langle}{\rangle}

\DeclarePairedDelimiterX{\menge}[2]{\{}{\}}{
      \,
      #1 
      \nonscript\,
      \delimsize\mid %
      \allowbreak
      \nonscript\, 
      \mathopen{} 
      #2 \,}

\DeclarePairedDelimiterX{\Jacobi}[2]{(}{)}{\frac{#1}{#2}}

\DeclareMathOperator{\Ker}{Ker}        
\DeclareMathOperator{\Aut}{Aut}
\DeclareMathOperator{\Z}{\mathbf{Z}}    
\DeclareMathOperator{\C}{\mathbf{C}}    
\DeclareMathOperator{\Irr}{Irr}
\DeclareMathOperator{\mat}{\mathbf{M}}      
\DeclareMathOperator{\enmo}{End} 
\DeclareMathOperator{\Hom}{Hom}
\DeclareMathOperator{\tr}{tr}
\DeclareMathOperator{\Sp}{Sp}
\DeclareMathOperator{\SL}{SL}
\DeclareMathOperator{\GL}{GL}

\DeclareMathOperator{\ann}{ann}
\DeclareMathOperator{\ord}{\mathbf{o}}
\DeclareMathOperator{\sign}{sign}
\DeclareMathOperator{\disc}{disc}
\DeclareMathOperator{\JacR}{\mathbf{J}}  
\DeclareMathOperator{\id}{id}

\usepackage{amsthm}    

\usepackage{csquotes}   

\usepackage{enumitem}   
\setlist{noitemsep}
\setlist[1]{labelindent=\parindent,leftmargin=*}
\setlist[enumerate,1]{label=\textup{(\alph*)}}

\usepackage{hyperref}
\hypersetup{colorlinks,linkcolor=blue,citecolor=blue}

\usepackage[capitalize, nameinlink]{cleveref}

\newtheorem{introthm}{Theorem}

\swapnumbers
\newtheorem{thm}{Theorem}[section]
\newtheorem{cor}[thm]{Corollary}
\newtheorem{lemma}[thm]{Lemma}
\newtheorem{prop}[thm]{Proposition}

\theoremstyle{definition}

\newtheorem{defi}[thm]{Definition}
\newtheorem{hypo}[thm]{Basic Setup}
\newtheorem{example}[thm]{Example}
\newtheorem{remark}[thm]{Remark}

\newcommand{\defemph}[1]{\textbf{#1}} 

\KOMAoptions{toc=flat,sectionentrydots=true}
\RedeclareSectionCommand[tocbeforeskip=0pt]{section}
\setkomafont{sectionentry}{\normalfont\normalsize}

\usepackage{lmodern}
\usepackage{microtype}

\makeatletter
\renewcommand*\author[1]{%
  \stepcounter{author}%
  \ifnum\c@author=1
    \gdef\@author{#1}%
  \else
    \xdef\@author{\unexpanded\expandafter{\@author\and#1}}%
  \fi
  \csgdef{author@\the\c@author}{#1}}
\newcommand*\email[1]{%
  \csgdef{email@\the\c@author}{#1}}
\newcommand*\address[1]{%
  \csgdef{address@\the\c@author}{#1}}
\AtEndDocument{%
  \xdef\author@count{\the\c@author}%
  \c@author=1
  \print@authors}
\newcommand*\print@authors{%
  \ifnum\c@author>\author@count
  \else
    \print@author{\the\c@author}%
    \advance\c@author by 1
    \expandafter\print@authors
  \fi}
\newcommand*\print@author[1]{%
  \par\medskip
  \noindent
  \begin{tabular}{@{}l@{}}%
    \csuse{author@#1}\\
    \csuse{address@#1}\\
    \textit{e-mail:}
    \texttt{\csuse{email@#1}}
  \end{tabular}
}

\renewcommand*{\title}[2][]{%
  \ifstrempty{#1}{
  \gdef\shorttitle{#2}\gdef\@title{#2}
  }{\gdef\shorttitle{#1}\gdef\@title{#2}%
  }
}
\makeatother

\RequirePackage{xpatch}

\RequirePackage{scrlayer-scrpage}

\ohead{\pagemark}
\makeatletter 
\cehead[]{\let\thanks\@gobble\@author}
\cohead[]{\shorttitle}
\makeatother
\ofoot[]{}

\makeatletter
\xpatchcmd{\swappedhead}%
         {~}%
         {.~}
         {}%
         {\typeout{failed to patch swappedhead}}
\makeatother

\KOMAoptions{twoside=semi,headings=small, numbers=enddot}

\setkomafont{disposition}{\normalcolor\bfseries}
\addtokomafont{title}{} 

\setkomafont{author}{\large}
\setkomafont{date}{\normalfont}

\newkomafont{abstract}{\normalfont\footnotesize}
\newkomafont{abstracttitle}{%
   \usekomafont{disposition}\usesizeofkomafont{abstract}%
}

\renewenvironment{abstract}%
  {\usekomafont{abstract}
   \begin{description}[%
        style=unboxed,
        leftmargin=3em,
        labelindent=\leftmargin,
        rightmargin=\leftmargin,
        labelwidth=0pt,        
        parsep=0pt,
        listparindent=\parindent, 
        font=\normalfont\usekomafont{abstracttitle}]
   \item[\abstractname.]}%
  {\end{description}}

\newcommand*{\keywordsname}{Keywords}
\newcommand*{\mscname}[1]{Mathematics Subject Classification (#1)}
  
\newcommand{\keywords}[1]{\item[\keywordsname.] #1.}  
\newcommand{\subjclass}[2][2010]{%
\item[\mscname{#1}.] #2.}

\newcommand*{\tpnewline}{\texorpdfstring{\\}{}}  
\newcommand*{\mytitle}{%
   The character of the Weil representation\tpnewline %
   of a finite abelian group of odd order}
\newcommand*{\myshorttitle}{The character of the Weil representation}
\newcommand*{\myauthor}{Frieder Ladisch}
\newcommand*{\mykeywords}{Weil character, Weil representation, 
symplectic groups, oscillator representation,
finite abelian groups}

\hypersetup{pdfauthor={\myauthor},
            pdftitle={\mytitle},
            pdfkeywords={\mykeywords}
            }

\title[\myshorttitle]{\mytitle}                           
\author{\myauthor}
\address{Universität Rostock\\
         Institut für Mathematik\\
         18051 Rostock\\
         Germany}
\email{frieder.ladisch@uni-rostock.de}

\date{}%

\begin{document}
\maketitle
\begin{abstract}
    Let $V$ be a finite abelian group of odd order,
    equipped with a non-degenerate, alternating form
    $\omega\colon V\times V \to \ints/m\ints$.
    We give closed formulas for the character values 
    of the Weil representation
    associated with $(V,\omega)$.
    These formulas generalize the ones given by
    S.~Gurevich and R.~Hadani (2007) and by 
    T.~Thomas (2008, 2013) from finite vector spaces
    to abelian groups.
    Our formulas do not involve the choice of a 
    Lagrangian subgroup of $V$.
    Our proofs are based on an elementary algebraic 
    approach introduced by H.~N.~Ward (1972, 2017)
    for finite vector spaces over fields.
\tableofcontents
\end{abstract}


\section{Introduction}
\subsection{Results}
Let $R$ be a finite, commutative ring of odd order
with some primitive additive character
$\lambda \colon (R,+)\to \compl^*$,
where \emph{primitive} means that no nonzero ideal is contained
in $\Ker \lambda$.
Let $V$ be a finite module over $R$,
with a symplectic form
$\omega \colon V\times V\to R$.
The symplectic group $\Sp(V)=\Sp_R(V,\omega)$ consists of the
$R$-linear automorphisms of $V$ preserving the form~$\omega$,
This group has a well known representation~$W$
called the \emph{Weil representation}, 
after a celebrated paper by André Weil~\cite{Weil64},
who studied analogous representations
 in the case
when $V$ is a locally compact abelian group.
In this paper, we give formulas for the character, $\tr W$, of the 
Weil representation~$W$
in the situation where $V$ is a finite module as described above.
(This includes the case that $V$ is a finite abelian group of odd 
order). 

Our first main result is the following:
\begin{introthm}
\label{ti:valuesodd}
    Let $g\in \Sp(V)$ have odd order.
    Then
    \[ \tr W(-g)= \tr W(-1_V)
       \quad\text{and}\quad
       \tr W(g) = \frac{1}{ \sqrt{ {\abs{V} } } }
                  \sum_{v\in V} \lambda\left( \tfrac{1}{2}
                                   \omega(v,vg) 
                                   \right)
       \,.
    \]
\end{introthm}
This result has a surprisingly simple proof
in the approach we use in this paper.
As mentioned in the abstract, this approach is essentially
due to H.~N.~Ward~\cite{ward72,Ward17}, who considered vector spaces
over finite fields.
In \cref{sec:sympl_alg,sec:weilrep},
we will give a fairly detailed and self-contained exposition 
of Ward's construction of the Weil representation,
extending it to the case of finite abelian groups.
We will also explain how to define 
\emph{the} (canonical) Weil representation~$W$
associated to $(V,\omega,\lambda)$,
independently of whether $\Sp(V)$ is a perfect group or not.

Ward's explicit construction immediately yields a
convolution formula for $\tr W$ (see \cref{p:mult}),
from which \cref{ti:valuesodd} follows easily
(see \cref{p:oddupsi,c:oddordervalues}).

In the case when the orders of $g$ and $V$ both are powers of the same
prime, the second formula in \cref{ti:valuesodd} was proved
by I.~M.~Isaacs~\cite[Theorem~6.1]{i73},
with a much longer proof.
In a different context, this formula appears in my
dissertation\footnote{When writing my dissertation
    (2008),
    I was not aware that the \enquote{magic character}
    appearing there and in Isaacs's work is in fact
    the character of a Weil representation.%
}~\cite[Corollary 4.34]{ladisch09diss}.

The formula for $\tr W(g)$ in \cref{ti:valuesodd} 
is in general not correct when 
$g\in \Sp(V)$ has even order. 
Before we introduce all the definitions necessary for
the general formula, 
we give another special case.
When $\alpha$ is or induces a permutation
of some finite set $X$,
then $\sign_X(\alpha)$ denotes the \emph{sign} or 
\emph{signature} of this permutation. 

\begin{introthm}
\label{ti:inv_value}
   Let $g\in \Sp(V)$ be such that 
   $1-g$ is invertible.
   Then
   \[ \tr W(g) = (-1)^{ \frac{ \sqrt{ \abs{V} } -1 }{2} }
                 \sign_V(1-g).
   \]
\end{introthm}

We collect some elementary properties of $\sign$
for automorphisms 
of finite $R$-modules 
in \cref{sec:sign}.
For example, when $R=F$ is a finite field of odd order and 
$\alpha\in\GL(n,F)$, 
then $\sign_{F^n}(\alpha)=1$
if and only if $\det(\alpha)$ is a square in $F^*$.
In view of this, \cref{ti:inv_value} generalizes a result
of S.~Gurevich and R.~Hadani~\cite[Theorem~2.2.1]{gurevichhadani07},
which was obtained using algebraic geometry.

For the general formula, we need some more notation.
For $g\in \Sp(V)$, define a bilinear form $B_g$ on 
$X=V(1-g)$ by
\[ B_g(v(1-g),w(1-g)) = \omega( v, w(1-g) )
   \quad \text{ for } v,w\in V.
\]
Then $B_g$ is well-defined and non-degenerate as form on
$X=V(1-g)$ (see \cref{l:sigmaform} below).

Second, 
for an arbitrary 
\emph{symmetric,} non-degenerate bilinear form
$q\colon X\times X \to R$, set 
\[  \gamma_{\lambda}(q) := \frac{1}{\sqrt{\abs{X}}} 
\sum_{x\in X} \lambda( \tfrac{1}{2}q(x,x)).
\]
We will prove some elementary properties 
of such normalized \emph{Gauss sums}
in \cref{sec:gausssums}, for example that
$\gamma_{\lambda}(q)^2 = (-1)^{(\abs{X}-1)/2}$.
(In the finite field case, 
$\gamma_{\lambda}(q)$ is the \emph{Weil index}
in the usual sense~\cite{Weil64,Perrin80}, 
but we do not need this.)

The following main result is proved in \cref{sec:values},
and \cref{ti:inv_value} is deduced as 
\cref{c:gm1_inv}.

\begin{introthm}
\label{ti:main}
    Let $g\in \Sp(V)$ and write $X=V(1-g)$.
    If
    $q\colon X \times X \to R$ 
    is a non-degenerate, symmetric form,
    then
    there is a unique 
    $\alpha\in \GL_R(X)$ such that
    $q(x,y) = B_g(x\alpha,y)$ for all
    $x$, $y\in X$.
    Then
    \[ \tr W(g) = \sqrt{\abs{\C_V(g)}} \:
    \sign_X(\alpha) \,\gamma_{\lambda}(-q).
    \]
\end{introthm}

When $R$ is a finite field of odd order, 
then $\sign_X(\alpha)=1$ in \cref{ti:main} 
if and only if
$q$ and $B_g$ have the same discriminant.
(In general, $B_g$ is not symmetric.)
It follows that in the finite field case,
the formula in \cref{ti:main} reduces to the one
developed by T. Thomas in his 
2013 paper~\cite[Corollary~1.4]{Thomas13}.
The formula in Thomas's earlier paper~\cite[Theorem~1A]{thomas08}
follows easily by elementary rules for the evaluation of Gauss sums
(cf. \cref{ex:field}).

At least when $R$ is a finite principal ideal ring,
for example, $R=\ints/ m \ints$, 
it is not difficult to see that 
there exist non-degenerate, symmetric forms
on any finite $R$-module.
In our proof of \cref{ti:main},
we will show that for a finite principal ideal ring $R$, 
there is such a form~$q$
for which the $\alpha$ in \cref{ti:main}
has even parity, that is, $\sign_X(\alpha)=1$,
and so 
$\tr W(g) = \sqrt{ \abs{ \C_{V}(g) } } \gamma_{\lambda}(-q)$.
(Thomas~\cite[p.~1538]{Thomas13} takes the viewpoint
that the form $B_g$, which is in general not symmetric, 
determines an element of the Witt ring modulo a certain ideal.
Thus $B_g$ determines an equivalence class of symmetric forms
and the Weil index $\gamma_{\lambda}$ is constant on this equivalence 
class.)
In general, we can always view $V$ and $X$ just as abelian groups,
and thus replace $R$ by
$\ints/m\ints$ with suitable $m$.

The Weil representation depends on the primitive linear character
$\lambda\colon R\to \compl^*$.
As a corollary of our results, 
we prove in \cref{c:GMT18} the following
nice formula, which relates the characers 
of the Weil representations
$W_{\lambda}$ and $W_{\lambda^2}$:
\[  \tr W_{\lambda}(g) \cdot \tr W_{\lambda}(-g)
    = (-1)^{ \frac{ \sqrt{\abs{V}}-1 }{2} } 
      \cdot \tr W_{\lambda^2}(g^2).
\]
For finite fields, this formula is due to
R.~Guralnick, K.~Magaard and P.~H.~Tiep
\cite[Theorem~1.2]{GuralnickMagaardTiep18}.
(We have $\tr W_{\lambda} = \tr W_{\lambda^2}$
on $\Sp_R(V)$
if and only if $2$ is a square in $R$.)

\subsection{Background}

Weil representations were introduced by A.~Weil \cite{Weil64}
for locally compact abelian groups, 
in particular vector spaces over local fields 
or adelic rings.
Weil representations for symplectic vector spaces over finite fields 
appeared probably first in a 1961 paper by
B.~Bolt, T.~G.~Room, and G.~E.~Wall~\cite{BoltRoomWall61},
independently of Weil's work.
Although Weil suggested that the case of finite abelian groups 
would merit closer investigation~\cite[p.~143--144]{Weil64},
most of the literature on Weil representations associated
to finite abelian groups is concerned with the case of vector spaces
over finite fields.

S.~Tanaka~\cite{Tanaka67a,Tanaka67b} 
used Weil representations associated to
$\ints/p^k\ints \oplus \ints/p^{\ell} \ints$ to construct the 
irreducible representations of $\SL(2,\ints/p^k\ints)$.
Weil representations
of finite abelian groups in general were
also studied by A.~Prasad~\cite{prasad09p}, K.~Dutta and A.~Prasad~\cite{DuttaPrasad15} and 
N.~Kaiblinger and M.~Neuhauser~\cite{KaiblingerNeuhauser09}.
By now, there is also an extensive literature 
on Weil representations of $\Sp(2n,R)$ and unitary groups over 
finite rings~$R$
\cite{climcnsze00,CruickshankGutierrezSzechtman20,gowszecht02,%
szechtman99t,Szechtman05}
(to name just a few references).

The absolute value of the character of the Weil representation
of a vector space over a finite field
was determind by R.~Howe \cite{howe73}.
Formulas for values of the Weil character of the symplectic group of a
vector space over a finite field
were given by P.~Gerardin~\cite{gerardin77}
and
K.~Shinoda~\cite{shinoda80},
but these formulas are quite complicated and 
partly depend on case analysis.
Formulas 
for the character values on a set of generating 
elements of the symplectic group were given by M.~Neuhauser
\cite{Neuhauser02},
using concrete matrices for the Weil representation.

At the same time as Howe, but in a completely different context,
the character of the Weil representation associated with a
finite abelian $p$-group was studied extensively by
I.~M.~Isaacs~\cite{i73}. 
(In Isaacs's paper, the term \enquote{Weil representation}
is not used. The relation of Isaacs's work to the Weil representation
and its character can perhaps best seen from his
Theorems~4.7 and~4.8.)
Isaacs proved that $\abs{\tr W(g)}^2 = \abs{\C_V(g)}$ 
\cite[Thms.~3.5(a), 4.8]{i73}, 
and gave an algorithm for determining the signs
of $\tr W(g)$ for all $g\in \Sp(V)$, 
without giving a closed formula.
He also proved our \cref{ti:valuesodd} under the assumption
that $g$ has $p$-power  order.

A quite elementary approach to the Weil representation
associated to a finite abelian group was given by
A.~Prasad~\cite{prasad09p}. 
Prasad used his methods
to give a very simple proof of the equality
$\abs{\tr W(g)}^2 = \abs{\C_V(g)}$
for arbitrary abelian groups.

In the case where $V$ is a vector space
over a finite field,
S.~Gurevich and R.~Hadani \cite{gurevichhadani07} 
found a simple formula for $\tr W(g)$
when
$g\in \Sp(V)$ is such that $g-1$ is invertible
(essentially the one in \cref{ti:inv_value} above).
Their proof uses techniques of algebraic geometry.
T.~Thomas~\cite[Theorem~1A]{thomas08} 
\cite[Corollary~1.4]{Thomas13} found 
simple formulas for
$\tr W(g)$ for arbitrary $g\in \Sp(V)$, 
where $V$ is a vector space over a 
(finite or local) field.
Thomas's approach works uniformly for finite and local fields, 
and uses machinery like the Weil index, the Maslov index and a 
construction of the metaplectic group as an extension
of the symplectic group by a subfactor group of the Witt group.
A more elementary, but rather longish proof of
Thomas's results was given by 
A.-M. Aubert and T. Przebinda~\cite{AubertPrzebinda14}.
In the finite field case, H.~N.~Ward~\cite{Ward17}
gave a short and elementary proof of a version of Thomas's 
character formula,
building on his approach to the Weil representation
from 1972~\cite{ward72}.
Our proof of \cref{ti:main} owes a significant debt to 
the ideas in Ward's preprint~\cite{Ward17}.
The methods of our proofs are rather elementary, 
and we have made an effort to 
make the paper self-contained.

\section{Preliminaries and Notation}
\label{sec:prelim}

Throughout, we will assume the following:

\begin{hypo}\label{basicsetup}
    $R$ is a finite commutative ring (with $1$)
        such that $2$ is invertible in $R$,
    and 
    $\lambda\colon R\to \compl^*$ 
    is a primitive additive character of $R$,
    where \emph{primitive}  means that
    $\Ker \lambda$ contains no nonzero ideal of $R$.
       (A finite ring $R$ has a primitive additive character
       if and only if $R$ is a finite Frobenius ring.
       We refer the reader to T.~Honold's paper~\cite{Honold01} and 
       the references therein.)
       
    Additionally, $V$ is a finite module over $R$
      and $\omega\colon V\times V \to R$ is a non-degenerate,
      bilinear alternating form on $V$. 
      We write $\Sp_R(V,\omega)$
      (or simply $\Sp(V)$ when $R$ and $\omega$ 
      are clear from context) 
      for the 
      group of $R$-linear automorphisms~$g$ of $V$
      preserving the form, that is,
      $\omega( vg , wg ) = \omega( v , w )$ for all
      $v$, $w\in V$.   

    (The case of a finite abelian group $V$ of odd order
is covered by 
$R=\ints/m\ints$, 
where $m$ is any odd multiple of the exponent of $V$.)
\end{hypo}

    If $B\colon U\times W\to R$ is a bilinear form on the product of
    two $R$-modules, 
    and if $X\leq U$ and $Y\leq W$,
    we write
    \begin{align*}
       X^{B} 
       &= \menge{w\in W}{ B(x,w)=0 \text{ for all } x\in X }
          \quad\text{and}\quad
       \\
       \leftidx{^{B}}{Y}{}
       &= \menge{ u\in U }{ B(u,y)=0 \text{ for all } y\in Y }.
    \end{align*}
    We call $B$ \emph{non-degenerate} when
    $U^B = 0$ and $\leftidx{^{B}}{W}{} = 0$.
    For $B=\omega$, we write 
    $X^{\perp}$ instead of $X^{\omega}$.

\begin{lemma}\label{l:perp_basic}
    Let $R$ be a finite commutative ring with a primitive 
    additive character $\lambda\colon R\to \compl^*$,
    and let
    $B\colon U\times W\to R$ be  a non-degenerate
    bilinear form.
    Then the maps $U\to \Hom_R(W,R)$ and $W\to \Hom_R(U,R)$
    induced by $B$ are isomoprhisms.  
    For $X$, $\tilde{X}\leq U$, we have:
   \begin{enumerate}
   \item \label{it:perp_via_mu}%
        $w\in X^{B} \iff 
        \lambda(B( x , w ))=1$ for all $x\in X$.        
   \item $\leftidx{^B}{(X^B)} = X$.
   \item \label{it:perp_sum}
        $(X+\tilde{X})^{B} = X^{B} \cap {\tilde{X}}^{B}$
       and $(X\cap \tilde{X})^{B} = X^{B} + {\tilde{X}}^{B}$.
   \item $\abs{U} = \abs{X}\abs{X^{B}} = \abs{W}$.
   \end{enumerate}
   Similar statements hold on the other side for 
   submodules $Y$, $\tilde{Y}\leq W$.
\end{lemma}
\begin{proof}
   We begin with \ref{it:perp_via_mu}.
   For $w\in W$ and $X\leq U$,
   the set
   $B ( X, w ) $ is an ideal of $R$ .
   As $\Ker(\lambda)$ contains no non-zero ideals,
   $\lambda( B ( X, w ) ) = 1$ implies $w\in X^{B}$. 
   The direction \enquote{$\implies$} is clear,
   so \ref{it:perp_via_mu} follows.
   
   In particular, we see that the multiplicative form
   $\mu(x,y) = \lambda(B( x , y ))$ induces
   injections $W \hookrightarrow\widehat{U}:= \Hom(U,\compl^*)$
   and $U \hookrightarrow \widehat{W}$.
   From $\abs{U}=\abs{\widehat{U}}$ and $\abs{W}=\abs{\widehat{W}}$,
   these must be isomorphisms.
   The isomorphism $W \to \widehat{U}$ is the composition
   of the injective maps 
   $w \mapsto B(\cdot, w) \mapsto \lambda\circ B(\cdot, w)$,
   so all these are isomorphisms.
   Similarly, we have an isomorphism
   $W/X^{B} \to \widehat{X}$ 
   for every submodule $X\leq U$.
   Thus $\abs{W} = \abs{X}\abs{X^{B}}$,
   and the rest follows easily.
\end{proof}

The following will often be used without further reference,
especially the case $r=s=1$:
\begin{lemma}
\label{l:ker_senkr}
   Assume \cref{basicsetup} and $g\in \Sp(V)$,
   and let $rs=1_R$ for $r$, $s\in R$.
   Then $\Ker(r-g) = (V(s-g))^{\perp}$.
\end{lemma}
\begin{proof}
   We have $\omega(v,w(s-g)) = \omega(v(s-g^{-1}),w)$.
   Thus the non-degeneracy of $\omega$ yields:
   $v\in (V(s-g))^{\perp}
    \iff 0 = v(s-g^{-1}) = v (g-r)sg^{-1}
    \iff v \in \Ker(r-g)$.
\end{proof}

Following C.~E.~Wall~\cite{Wall63} and 
T.~Thomas~\cite{thomas08,Thomas13},
we introduce a bilinear form on $V(1-g)$ for any 
$g\in \Sp(V)$.

\begin{lemma}\label{l:sigmaform}
    Assume \cref{basicsetup} and let $g\in \Sp(V)$. Define
    \[ B_g (v(1-g),w(1-g)) = 
    \omega ( v,w(1-g) ).
    \]
    Then
    \[ B_g\colon V(1-g)\times V(1-g) 
    \to R
    \]
    is a well-defined, non-degenerate bilinear form
    with
    $B_g(x,y) -B_g(y,x) = \omega(x,y)$
    for all $x$, $y\in V(1-g)$.
\end{lemma}

We should note here 
that $B_g(x,y) = - \sigma_g(x,y) = - \Theta_g(y,x)$,
where $\sigma_g$ is the form used by T.~Thomas \cite{Thomas13}, 
and $\Theta_g$ is H.~N.~Ward's \cite{Ward17} \enquote{theta form}.
Our convention follows C.~E.~Wall~\cite{Wall63}.

\begin{proof}[Proof of \cref{l:sigmaform}]
    (cf.\ \cite[Lemma~1.1.1, Eq.~1.1.3]{Wall63}, 
    \cite[Lemma~3.4]{Ward17})
    Set $\alpha=1-g$, that is, $g=1-\alpha$. 
    Then the equation
    $\omega (vg , wg)  = \omega( v , w )$ translates to
    \[ \omega( v\alpha , w\alpha )
       = \omega( v, w\alpha ) + \omega( v\alpha, w ).
    \]
    Thus $v_1\alpha=v_2\alpha$ implies 
    $\omega(v_1,w\alpha) = \omega(v_2,w\alpha)$,
    so $B_g$ is well-defined.
    Also, if $w\alpha$ is such that
    $0 = B_g(v\alpha,w\alpha)= \omega(v,w\alpha)$ 
    for all $v\in V$, then
    $w\alpha=0$ by non-degeneracy of $\omega$.
    Thus $B_g$ is non-degenerate.
    Finally, from
    \begin{align*}  
    \omega( v\alpha , w\alpha )
    &= \omega( v, w\alpha ) + \omega( v\alpha, w )
    \\ &= \omega( v, w\alpha ) - \omega( w, v\alpha )
    \\ &= B_g( v\alpha, w\alpha ) 
    - B_g( w\alpha, v\alpha ),
    \end{align*}
    we see $B_g(x,y)-B_g(y,x)=\omega(x,y)$ for 
    $x$, $y\in V\alpha = V(1-g)$.
\end{proof}

When $B\colon X\times X\to R$ is a bilinear form
and $2$ is invertible in $R$,
then we can write $B = B_s + B_a$ 
with $B_s(x,y) = (1/2)\big( B(x,y) + B(y,x) \big)$
symmetric and
$B_a(x,y) = (1/2)\big( B(x,y) - B(y,x) \big)$ 
alternating.
We have seen $(B_g)_a = \omega/2$
(on $V(1-g)$).
In \cref{sec:values},
we need the following facts about the symmetric part
(also known as \emph{Cayley form}~\cite{Thomas13}):

\begin{lemma}\label{l:symformqg}
    Set $Q_g(x,y)= (1/2)\big( B_g(x,y) + B_g(y,x)\big)$.
    Then
    \[ Q_g(x,y)    
        = \frac{1}{2} ( \omega ( vg ,w ) + \omega ( wg ,v ) )
        = \frac{1}{2} B_g(x(1+g),y)
    \]
    for $x=v(1-g)$, $y=w(1-g)$.
    The radical of $Q_g$ is $\Ker(1+g)$.
\end{lemma}
\begin{proof}
    The formulas for $Q_g$ can be verified by straightforward
    calculations.
    Notice that $v= (1/2)v(1-g) + (1/2)v(1+g)$ for all $v\in V$,
    and thus $\Ker(1+g)\subseteq V(1-g)$.
    The claim on the radical follows from
    $B_g$ being non-degenerate.
\end{proof}

\section{The symplectic algebra}
\label{sec:sympl_alg}

In this section, we extend definitions and results
of  Ward~\cite{ward72,Ward17} from the situation
where $R$ is a finite field to our more general situation.
The proofs are essentially the same.

\begin{defi}
   Assume \cref{basicsetup} and let 
   $\crp{K}\subseteq \compl^*$ 
   be a field containing the values of 
   $\lambda$.
   The \defemph{symplectic algebra} 
   $\cA = \cA(\crp{K},V, \lambda\circ\omega)$ 
   is the twisted group ring of $V$ over $\crp{K}$
   with factor set $\lambda\circ(\frac{1}{2}\omega)$.
   This means that $\cA$ has a $\crp{K}$-basis 
   \[   \menge{ b_v }{ v\in V }
   \]
   indexed by $V$, and multiplication is given by
   \begin{equation}  
      b_vb_w = \lambda(\tfrac{1}{2}\omega( v , w ))\, b_{v+w}
             = (\lambda\circ\omega)
                 \left( v/2,w \right)\, b_{v+w}.
   \label{equ:multbv}
   \end{equation}
   (Since $2$ is invertible in $R$, this makes sense.)
\end{defi}

The factor $\frac{1}{2}$ is not present in Ward's definition.
We have introduced this factor 
for consistency with the usual definition of 
the Heisenberg group.
The  \defemph{Heisenberg group}
$H=H(V,R,\omega)$
is the set
$V\times R$ with  multiplication defined by
\[(v,r)(w,s) = (v+w, r+s+ \frac{1}{2}\omega( v , w )).\]
It is routine to verify that $H$ is a group with this
multiplication,
with center $\Z(H) = 0\times R\iso R$.
Moreover, 
$(v,r) \mapsto b_v \lambda(r)$ is a group homomorphism
from $H$ into the unit group of $\cA$,
with kernel $\Ker \lambda$.
As we will see below,
$\cA$ is isomorphic to a matrix ring over $\crp{K}$.
The resulting representation of $H$ is known as the
\emph{Schrödinger representation of type $\lambda$,}
and is the unique irreducible representation of $H$
lying over $\lambda$.
These facts allow to translate between the approach taken here
and others like the one by A.~Prasad~\cite{prasad09p}.

\begin{lemma}
\label{l:symplalgformulas}
    The algebra $\cA$ has the following properties:
    \begin{enumerate}
    \item $b_0 =1_{\cA}$ is the identity.
    \item Every $b_v$ is a unit with inverse
         $(b_v)^{-1}= b_{-v}$.
    \item $b_w^{-1} b_v b_w = \lambda(\omega( v , w ))\, b_v$.
    \label{it:conjbv}
    \item \label{it:symplalgcenter}
          The center of $\cA$ is
          $\Z(\cA) = \crp{K} 1_{\cA}$.
    \end{enumerate}
\end{lemma}
\begin{proof}
    Easy calculations.
    \ref{it:symplalgcenter} follows from
    \ref{it:conjbv} 
    and the non-degeneracy of $\lambda\circ\omega$.
\end{proof}

\begin{prop}[cf.~{\cite[Theorem~1.3]{ward72}}]
\label{p:sympl_alg_matring}
   The symplectic algebra $\cA$ is isomorphic 
   to the  
   $n \times n$ matrix algebra over $\crp{K}$,
   where $n^2= \abs{V}$.
\end{prop}
\begin{proof}
   Let $L\leq V$ be a Lagrangian submodule,
   that is, 
   a submodule with $L^{\perp} = L$.
   (Notice that any submodule maximal subject to
    $L\subseteq L^{\perp}$ is Lagrangian.)   
   Then $\abs{V} = \abs{L}^2$
   by \cref{l:perp_basic}.
   Set 
   \[ e:= \frac{1}{\abs{L}} \sum_{x\in L} b_x.
   \]
   Then $eb_x=b_xe=e$ for $x\in L$ and $e^2 = e$.
   For $t\notin L$, we have
   \begin{align*}
     eb_te &= \frac{1}{\abs{L}^2} 
              \sum_{x,y\in L} b_xb_tb_y
       \\  &= \frac{1}{\abs{L}^2} 
              b_t \sum_{x,y\in L} \lambda(\omega( x , t ))b_xb_y
           && \text{(\cref{l:symplalgformulas}\ref{it:conjbv})}
        \\ &= \frac{1}{\abs{L}^2} 
                b_t \sum_{z,x\in L} \lambda(\omega( x , t ))b_{z}
            =0,
   \end{align*}
   where the last equality follows since
   $x\mapsto \lambda(\omega ( x ,t ))$
   is a nontrivial character of $L$ for $t\notin L$.
   
   Now let $T$ be a set of coset representatives of $L$ in $V$,
      that is,
      $V = \bigdcup_{t\in T } (L+t)$,   
   and form the set of elements
   \[ e_{st}:= b_s^{-1}e b_t
      , \quad s,t\in T.
   \]
   From the above it follows that
   $e_{st}e_{qr}=\delta_{tq}e_{sr}$.
   Moreover, using \cref{l:symplalgformulas}\ref{it:conjbv},
   we get
   \begin{align*}
     \sum_{t\in T} e_{tt}
     &= \sum_{t\in T} \frac{1}{\abs{L}} \sum_{x\in L} b_t^{-1}b_x b_t
     = \frac{1}{\abs{L}} 
        \sum_{x\in L} \sum_{ t\in T } 
       \lambda(\omega(x,t)) b_x
      = b_0 = 1_{\cA}.
   \end{align*}   
   (Here we use that
    $t\mapsto \lambda(\omega(x,t))$ is a nontrivial character
    of $V/L$ for $x\neq 0$.)
   Therefore, the elements $ e_{st}$ for $s$, $t\in T$
   form a full set of matrix units,
   and their $\crp{K}$-span is a 
   $\abs{T}\times \abs{T}$ matrix ring over $\crp{K}$
   \cite[17.4, 17.5]{lamMR}.
   Since $e_{ss} b_v = \mu(s,v) b_s^{-1}eb_t =  \mu(s,v) e_{st}$ 
   for some
   $\mu(s,v) \in \crp{K}$ 
   and the $t\in T$ with $L+s+v=L+t$,
   we see that $b_v = \sum_s e_{ss}b_v$ is in the $\crp{K}$-span 
   of the $e_{st}$'s
   and so the matrix units span $\cA$.   
   As $\abs{V}= \abs{L}\abs{T}= \abs{T}^2$, the proposition follows.
\end{proof}

\begin{remark}
    The last result and its proof remain valid when
    $\crp{K}$ is a ring such that $\abs{V}$ is invertible in $\crp{K}$
    and such that there exists a primitive character 
    $\lambda \colon R \to \crp{K}^*$.
\end{remark}

Since $\cA $ is a matrix ring, it has a trace function
$\tr\colon\cA \to \crp{K}$.
We need the trace on the canonical basis.
\begin{lemma}\label{l:tracesx}
    $\tr b_v = n\delta_{v,0}$, where $n=\sqrt{ \abs{V} } $.
\end{lemma}
\begin{proof}
    $\tr b_0= n$ is clear since $b_0=1_{\cA}$.
    For $v\neq 0$, there exists $w\in V$
    with $\lambda(\omega(v,w))\neq 1$.
    Thus $\tr b_v = 0$ for $v\neq 0$
    follows from \cref{l:symplalgformulas}\ref{it:conjbv}.
\end{proof}

\section{The Weil representation}
\label{sec:weilrep}

There is a natural action of $\Sp_R(V,\omega)=\Sp(V)$
on $\cA$,
namely
\[ \left(\sum_{v\in V} c_v b_v \right)^g  = 
   \sum_{v\in V} c_v b_{vg}
    \quad (c_v\in \crp{K}).
\]
The usual construction of the Weil representation is as follows:
Since $\cA \iso \mat_n(\crp{K})$, 
there exists, for any $g\in \Sp(V)$, an invertible element
$P(g)\in \cA$ (unique up to scalars), such that
$b_{vg} = b_v^{P(g)}$ for all
$v\in V$.
We can view
$P\colon \Sp(V)\to \cA^*\iso \GL(n,\crp{K})$
as a projective representation,
and one can show that this can be made into an honest
representation.

In this paper, we take a more constructive approach,
which is essentially due to 
H.~N.~Ward \cite[Proposition~2.1]{ward72},
who considers the case where $R$ is a finite field.
The same idea works in our more
general situation.
We use the form $B_g$ introduced in \cref{l:sigmaform}.

\begin{thm}\label{p:weilconcret}
  For $g\in \Sp(V)$, 
  \begin{equation} 
  \label{equ:WardsP}
     P(g)
      = \sum_{x\in V(1-g) } 
            \lambda(\tfrac{1}{2} B_g(x,x)) 
            \cdot b_x  \in \cA
  \end{equation} 
  is invertible 
  and we have
  $(b_v)^{P(g)} = b_{vg}$ for all $v\in V$.
  For $g$, $h\in \Sp(V)$, we have 
  $P(g)P(h)=c(g,h)P(gh)$ with
  \[ c(g,h) = \sum_{x\in V(1-g)\cap V(1-h)}
              \lambda\left( \tfrac{1}{2}
              ( B_g(x,x) +  B_h(x,x) 
              ) 
              \right) .
  \]
\end{thm}
\begin{proof}
  Let $v\in V$.
  Then
  \begin{align*}
    b_v \cdot b_{v}^{-g}
      &= b_v \cdot b_{-vg} 
      = \lambda( \tfrac{1}{2}\omega ( v , -vg ) ) \cdot b_{v-vg}
       = \lambda( \tfrac{1}{2}B_g(x,x)) \cdot b_x
  \end{align*}
  for $x=v(1-g)$.
  This shows that
  \begin{equation} 
    P(g) = \frac{1}{\abs{\C_V(g)}} 
            \sum_{v\in V} b_v b_{v}^{-g}.
  \label{equ:WardsPohneBg}
  \end{equation}
  It follows
  \begin{align*} 
       b_w \cdot P(g) \cdot b_{w}^{-g}
       &= \frac{1}{\abs{\C_V(g)}}
          \sum_{v\in V} \lambda(\tfrac{1}{2}\omega(w,v))
                        \,b_{w+v} 
                        \,
                        \,\big( \lambda( \tfrac{1}{2}\omega(w,v)
                                       )
                              b_{w+v}
                          \big)^{-g}
      \\ &= \frac{1}{\abs{\C_V(g)}}
            \sum_{v\in V} b_{w+v} b_{w+v}^{-g} 
          = P(g)
  \end{align*}
  and thus
  $ b_w \cdot P(g) P(g^{-1}) = P(g) P(g^{-1}) \cdot b_w$ 
  for all $w\in V$.
  Therefore, 
  $ P(g) P(g^{-1}) \in \Z(\cA ) = \crp{K}b_0$.
  Thus to show that $P(g)$ is invertible, it suffices to show
  that the coefficient of $b_0$ in
  $P(g)P(g^{-1})$ is not zero.  
  
  More generally, for $g$, $h\in \Sp(V)$,
  let $c(g,h)$ be the coefficient of $b_0$ in
  $P(g)P(h)$. 
  Then
   \begin{align*}
      c(g,h) &=
         \sum_{x\in V(1-g) \cap V(1-h)}
          \lambda( \tfrac{1}{2} B_g(x,x) ) 
          \lambda( \tfrac{1}{2} B_h(-x,-x) )
          \lambda( \tfrac{1}{2} \omega(x,-x) )
      \\ &= \sum_{x\in V(1-g)\cap V(1-h)}
             \lambda\left( \tfrac{1}{2}
                         ( B_g(x,x) +  B_h(x,x) ) 
                    \right).
    \end{align*}
    For $h=g^{-1}$, the equality 
    $v(1-g) = -vg(1-g^{-1})$ yields   
   $V(1-g) = V(1-g^{-1})$,
   and 
   $ B_{g^{-1}}(x,x) = -B_g(x,x)$ for all $x\in V(1-g)$.
   Thus $c(g,g^{-1}) = \abs{ V(1-g)} \neq 0$,
   and $P(g)$ is invertible.
   
   Since also 
   $P(g)P(h)P(gh)^{-1}\in \Z(\cA)=\crp{K}b_0$
   and the coefficient of $b_0$ in $P(gh)$ is $1$,
   we see that $P(g)P(h)=c(g,h)P(gh)$ as claimed. 
   The proof is complete.
\end{proof}

Although it is well known,
even in this generality~\cite{climcnsze00,i73,prasad09p}, 
we give here the simple proof 
that $P$ can be made into an honest representation.
We give the proof because
at the same time we define
\enquote{the} \emph{canonical} Weil representation.

\begin{lemma}\label{l:ccinv}   
   Set
   \[ T
        = \frac{1}{ \sqrt{\abs{V}} } 
          \sum_{v\in V} b_v.
   \]
   Then $T^2 = 1_{\cA}$, $\tr T = 1$ and
   $TP(g)=P(g)T$ for all $g\in G$.
\end{lemma}
\begin{proof}   
    That $T^2 = 1_{\cA}$ 
    follows from a straightforward computation
    in $\cA$, 
    or alternatively from
    $T = \left(1/\sqrt{\abs{V}} \right) P(-\id_V)$ 
    and
    \cref{p:weilconcret}.
    For the trace, use \cref{l:tracesx}.
    Finally,
    $T^{P(g)} = T$ is clear from
    $(b_v)^{P(g)} = b_{vg}$.
\end{proof}

By \cref{p:sympl_alg_matring},
$\cA$ is a matrix ring. 
Let $M$ be a simple $\cA$-module,
so $M\iso \crp{K}^{\sqrt{\abs{V}}}$. 
Let $E_{\pm}$ be the eigenspaces of $T$ in $M$
for the eigenvalues
$\pm 1$. 
Then $M = E_{+} \oplus E_{-}$.
When $c\in \cA$  centralizes $T$,
then $c$ maps $E_{\pm}$ into itself.
In particular, the $E_{\pm}$ are invariant under
$P(g)$ for all $g\in \Sp(V)$.

\begin{thm}\label{t:weilexist}
    Let $T$, $M$ and $E_{\pm}$ be as above.
    Then for any $g\in \Sp(V)$, there is a unique
    $W(g)\in \cA^*$ such that
    \begin{enumerate}
    \item \label{it:weilop}
          $ (b_v)^{W(g)} = b_{vg}$ for all $v \in V$, and
    \item \label{it:weildetpm}
          $\det W(g)_{|E_{+}} = \det W(g)_{|E_{-}}$.
    \end{enumerate}
    Moreover,
    $W(gh)=W(g)W(h)$ for all $g$, $h\in \Sp(V)$.
\end{thm}

\begin{proof}
    Let $d_{\pm} = \dim_{\crp{K}} E_{\pm}$.
    Then $d_+-d_- = \tr T = 1$.
    For $c\in \cA^*$ with $cT=Tc$, 
    consider 
    $\eta(c):= \det c_{|E_{+}} \cdot (\det c_{|E_{-}})^{-1}$.
    For a scalar $\mu$, we have
    $\eta(\mu c) = \mu^{d_+ - d_-} \eta(c) = \mu \eta(c)$. 
    It follows that $\eta(\eta(c)^{-1}c) = 1$
    and that this is the unique scalar multiple of $c$
    for which $\eta$ is $1$.
    
    Let $P(g)$ be as in \cref{p:weilconcret}.
    Then $W(g):= \eta(P(g))^{-1}P(g)$
    is the unique element in $\cA^*$ satisfying
    \ref{it:weilop} and \ref{it:weildetpm}.
    Since 
    $\eta\big( W(g) W(h) \big) 
     = \eta(W(g))\,\eta(W(h))=1$,
    we must have $W(g)W(h)=W(gh)$.
\end{proof}

We call $W=W_{\omega,\lambda}$ 
   \defemph{\emph{the} (canonical) Weil representation}   
   associated to $\omega$, $\lambda$.
   (Notice that $V$ and $R$ are implicit in 
   $\omega$ and $\lambda$.)
   Since 
   $\cA \iso \enmo_{\crp{K}}(M) \iso \mat_{\sqrt{\abs{V} }}(\crp{K})$,
   we can view $W\colon \Sp(V) \to \cA^*\iso \GL_{\crp{K}}(M)$
   as a representation in the usual sense.      
   The character $\psi=\psi_{\omega,\lambda}$ of $W$ is called 
   \defemph{the (canonical) Weil character.}

The purpose of this paper is to find a formula for
$\psi_{\omega, \lambda}(g)$.
The point of the next remark is that it is no loss
of generality to assume $R=\ints/m\ints$.
(The situation would be different if we were to study
the decomposition of $\psi$ into
irreducible characters, 
as in \cite{climcnsze00,DuttaPrasad15,prasad09p}.)

   \begin{remark}
   \label{rm:indepR}
       Since $\lambda(R)\subseteq \compl^*$
       is a finite cyclic subgroup of order $m$ (say),
       we can write 
       $\lambda = \lambda'\circ\kappa$,
       where
       $\kappa\colon R\to R':=\ints/m\ints$
       is a surjective group homomorphism
       and
       $\lambda'$ 
       a faithful linear character of $(R',+)$.
       We get the following commutative diagram:
       \[
           \begin{tikzcd}[row sep=tiny]
              & R \arrow[dd,"\kappa"] 
                  \arrow[dr,"\lambda"]
              & 
              \\
              V \times V \arrow[ur, "\omega"]
                         \arrow[dr, "\omega'"']
              & & \compl^* 
              \\
              & R' \arrow[ur,"\lambda'"'] &
           \end{tikzcd}
       \]
       In general,
       $\kappa$ is not a ring homomorphism,
       because $\Ker\lambda$ contains no non-zero ideal of $R$,
       but the composed form $\omega'=\kappa \circ\omega$
       is bi-additive and thus $R'$-bilinear, as $R'=\ints/m\ints$.
       Thus $(V,R',\omega',\lambda')$ satisfies \cref{basicsetup}.
       The symplectic algebra depends only on the 
       bilinear form (bicharacter)
       $\lambda\circ\omega \colon V\times V\to \compl^*$.
       Clearly,
       \[ \Sp_R(V,\omega)
          \subseteq \Sp_{R'}(V,\omega')
          = \Sp_{\ints}(V,\lambda\circ\omega).
       \]
       For $g\in \Sp_R(V,\omega)$,
       we have $W_{\omega',\lambda'}(g) = W_{\omega,\lambda}(g)$.
       This follows easily 
       by first observing that the analogous result 
       for the $P$ defined in \cref{p:weilconcret} holds.
   \end{remark}

    
\section{The Weil  character: simple properties}
\label{sec:weilcharI}

In this and the next section,
we assume \cref{basicsetup}, 
and we let $W\colon \Sp(V)\to \cA$ 
be \emph{the} canonical Weil representation associated 
to the data $(V,R,\omega,\lambda)$,
and $\psi = \tr W$ its character.

In the next result, $E_{\pm}$ are as defined before 
\cref{t:weilexist}.
We also write simply $-1$ for the central involution
in $\Sp(V)$ sending $v$ to $-v$.

\begin{prop}
    \label{p:psitaupm}
    Let $\psi_{+}$ and $\psi_{-}$ be the characters of $W$
    on $E_{+}$ and $E_{-}$, respectively.
    Then
    \begin{align*}
        \psi(-1) 
        &=(-1)^{\frac{\sqrt{\abs{V}}-1}{2}},
        \\
        \psi_{+}(g) 
        &= \frac{ \psi(g) + \psi(-1)\psi(-g)
        }{2}
        \quad  \text{and} 
        &  
        \psi_{-}(g) 
        &= \frac{ \psi(g) - \psi(-1)\psi(-g)
        }{2}.
    \end{align*}
\end{prop}
\begin{proof}
    We use the notation from the proof
    of \cref{t:weilexist}:
    It follows from that proof that
    $W(-1) = (\det T_{|E_{+}}) (\det T_{|E_{-}})^{-1} \cdot T
     = (-1)^{d_{-}} T$.
    Thus $\psi(-1) = (-1)^{d_{-}} \tr T = (-1)^{d_{-}}$.
        From $d_{+}+d_{-}=\dim M = \sqrt{\abs{V}}$ and
        $d_{+}-d_{-} = \tr T = 1$ it follows that
        $d_{-}= (\sqrt{\abs{V}}-1)/2$.
    
    The projection of $M$ to $E_{\pm}$ is given by
    $e_{\pm} = (1 \pm T)/2 = (1 \pm \psi(-1)W(-1))$.
    From this the other formulas follow.
\end{proof}

By using Ward's concrete formula for a projective
representation equivalent with $W$,
we get a convolution formula for $\psi$:

\begin{prop}\label{p:mult}
    Let 
    $g$, $h\in \Sp(V)$. Then
    \begin{align*}
    \psi(gh)
    &= \frac{ \psi(g)\psi(h) }{ \sqrt{\abs{V}} }
    \sum_{x\in V(1-g)\cap V(1-h)}
    \lambda\left( \tfrac{1}{2}
    ( B_g(x,x) +  B_h(x,x) 
    ) 
    \right).
    \end{align*}
\end{prop}
\begin{proof}
    Let $P(g)$ be as defined in \cref{p:weilconcret}.
    Then $W(g)=\mu(g)P(g)$ for some 
    $\mu(g)\in \crp{K}$ 
    (in fact, we computed $\mu(g)$ in the proof of \cref{t:weilexist}),
    and by taking traces, we see that we must have
    $\psi(g) = \mu(g) \sqrt{\abs{V}} $.
    From
    $W(gh)=W(g)W(h)$ and
    $P(g)P(h)=c(g,h)P(gh)$ it follows
    that
    \[ \frac{ \psi(gh) }{ \sqrt{\abs{V}} }
        =
        \frac{ \psi(g) }{ \sqrt{\abs{V}} }
        \frac{ \psi(h) }{ \sqrt{\abs{V}} } c(g,h).
    \]
    From the value of $c(g,h)$ in 
    \cref{p:weilconcret},
    the result follows.
\end{proof}

For the sake of completeness, we give a simple
proof of the following result, although it is
well known~\cite{howe73, i73, prasad09p}.
\begin{prop}\label{p:absvalue}
    We have
    \[ \abs{ \psi(g) }^2
    = \abs{\Ker(g-1)}=\abs{\C_V(g)}
     \quad \text{for all $g\in \Sp(V)$.}
    \]
\end{prop}
\begin{proof}
    The linear map $\cA \to \cA$
    sending $a$ to $a^{W(g)}$
    has trace
    \[ \tr W(g)^{-1} \tr W(g) = \abs{ \psi(g) }^2,
    \] 
    as is seen when using
    a set of matrix units as basis of $\cA$.
    On the other hand, 
    $\menge{ b_v }{ v\in V }$
    is a basis of $\cA$, and 
    $b_{v}^{W(g)} = b_{vg}$.
    The result follows.
\end{proof}

\begin{prop}\label{p:rational}
\hfill
  \begin{enumerate}
  \item \label{it:detWeil} 
        The order of the character
        $g\mapsto \det W(g)$        
        divides the order of $\lambda$.
  \item \label{it:rational}
       Suppose that the order of $g\in \Sp(V)$ 
       is prime to $\abs{V}$.
        Then $\psi(g)$ is rational.  
  \end{enumerate}  
\end{prop}
\begin{proof}
  Our results so far are valid 
  over the field
  $\crp{K} :=\rats(\lambda)$ generated by the values of
  $\lambda$.
  It follows that $\tr W(g)$, 
  $\det W(g)$ and
  $\det W(g)_{|E_{\pm}} \in \crp{K}$
  (where $E_{\pm} =\Ker(I\mp T)$ as defined before
   \cref{t:weilexist}).
  This already shows that 
  the order of $\det W(g)_{|E_{+}}$ divides $2\ord(\lambda)$.
  As 
  $\det W(g)= (\det W(g)_{|E_{+}})^2$, 
  \ref{it:detWeil} follows.  

  Now suppose that $\gcd(\ord(g),\abs{V})=1$.
  The eigenvalues of $W(g)$ 
  lie in a field $\crp{L}$ obtained by
  adjoining a primitive 
  $\ord(g)$\nobreakdash th 
  root of unity to $\rats$,
  and thus $\tr W(g)\in \crp{L}$.
  Since $\gcd(\ord(g),\abs{V})=1$, 
  we have
  $\crp{K} \cap \crp{L} = \rats$~\cite[Corollary on p.~204]{langAlg}.
  Thus $\psi(g)=\tr W(g)\in \rats$ as claimed.
\end{proof}

Another proof of \ref{it:detWeil} is by showing that
$\Sp(V)$ is generated by elements of order
dividing the order of $\lambda$
(which coincides with the characteristic of $R$,
since $\lambda$ is primitive).
In fact, it seems that $\Sp(V)$
is perfect with only certain exceptions,
so that $\Sp(V)$
either has no nontrivial linear characters
or only ones of order $3$.
(For example, for $V= (\ints/9\ints)^2$
 or $V= (\ints/9\ints \times \ints/3\ints)^2$,
 the symplectic group $\Sp(V)$ is not perfect.) 
W.~Klingenberg \cite[Corollary to Theorem~3]{Klingenberg63}
    has shown that $\Sp_R(V)$ is a perfect group
    when $V\iso R^{2n}$ and $R$ is a local ring with 
    $\abs{R/\JacR(R)}>3$. 
    His proof can be extended to the case 
    where $\abs{R/\JacR(R)}=3$ and $n>1$.

The next result,
together with the facts that
$\psi(g)^2=\abs{\C_V(g)}$ and $\psi(g)$ is rational,
allows to compute $\psi(g)$ inductively
for elements of $2$-power-order.
We need only the case $g^2 =1$, which can also be
proved directly.

\begin{prop}\label{p:2powervalues}
    Let $g\in \Sp(V)$ have order $2^k$.
    Then 
    \[\psi(g)\equiv 
      (\psi_{\erz{g^2}}, 1_{\erz{g^2}})_{\erz{g^2}}
    \mod 4.
    \]
\end{prop}
(Here, 
$(\alpha,\beta)_H 
   = (1/\abs{H}) \sum_{h\in H} 
        \alpha(h) \cconj{\beta(h)}
$ 
denotes the usual inner product for class function 
on a finite group~$H$, and $1_H$ the trivial character of $H$.)
\begin{proof}[Proof of \cref{p:2powervalues}]
    The eigenvalues of $W(g)$ are
    $2^k$-th roots of unity.
    Since $W(g)$ as matrix can be realized over $\rats(\lambda)$,
    the roots of unity $\eps$ and $\eps^{-1} = \cconj{\eps}$ occur
    with the same multiplicity as eigenvalue in $W(g)$, 
    and except for $\eps = \pm 1$, also 
    $\eps$ and $-\eps$ occur with the same multiplicity.
    It follows that
    \[ \psi(g)=\tr W(g) = m_1 + m_{-1}
       \quad \text{and}\quad
       \det W(g) = (-1)^{m_{-1}},
    \]
    where $m_{\pm 1}$ are the multiplicities of the eigenvalues
    $\pm 1$.
    By \cref{p:rational}~\ref{it:detWeil},
    we have $\det W(g) = 1$, so 
    $m_{-1}$ must be even.
    Thus
    \[ \psi(g) =  m_1-m_{-1} \equiv m_1 + m_{-1} \mod 4.
    \]
    But $m_1 + m_{-1}$ is the multiplicity of $1$
    as eigenvalue of $W(g^2)$.
    This is the multiplicity of the trivial character of
    $\erz{g^2}$ as a constituent of the restricted
    character $\psi_{\erz{g^2}}$.
    By the orthogonality relations of character theory,
    the result follows.
\end{proof}

\begin{cor}
\label{c:value_inv}
  Let $t\in \Sp(V)$ be an involution
    and set $c = \sqrt{ \abs{ \C_V(t) } }$,
    $d = \sqrt{ \abs{V(1-t)} } $. 
    Then $c$ and $d$ are positive integers
    and
    \[ 
      \psi (t)
         = (-1)^{\frac{d-1}{2}} c .
    \]
\end{cor}
\begin{proof}
    By  \cref{p:absvalue} and \cref{p:rational}, 
    $\psi(t) = \pm c \in \rats$.
    By \cref{p:2powervalues} with $k=1$, it follows that
    $\psi(t) \equiv \psi(1)=cd \mod 4$, 
    and thus the result.
\end{proof}

\begin{example}\label{ex:DFTelem}
    Let $g\in \Sp(V)$ be an element with 
    $g^2= -1$.
    (The Weil representation of such an element
    encodes the discrete Fourier transformation.)
    By \cref{p:2powervalues},
    we have 
    $\psi(g)\equiv (\psi(1)+\psi(-1))/2 \mod 4$.
    It follows that
    $\psi(g)=1$ if $\psi(1)\equiv 1,3 \mod 8$ and
    $\psi(g)=-1$ if $\psi(1)\equiv 5,7 \mod 8$, or,
    more succinctly,
    $\psi(g) = \Jacobi*{-2}{\psi(1)}$
    (Jacobi-Symbol).
\end{example}

\section{Values on elements of odd order}
\label{sec:oddorder}

We keep the notation introduced in the last section,
in particular in \cref{p:psitaupm}.
The ideas in the next proof are taken from
Isaacs's paper~\cite[Theorem~5.3]{i73}.

\begin{prop}\label{p:oddupsi}
  For any $g\in \Sp(V)$ of odd order,
  \[
  \psi(-g)= \psi(-1)
  \quad\text{and}\quad 
  \psi_{+}(g)-\psi_{-}(g)=1.
  \]
\end{prop}
\begin{proof} 
  Recall that 
  $\psi_{\pm}$ is the character of $W$ on $E_{\pm}$.
  It follows from the formulas in \cref{p:psitaupm}
  that $\psi_{+}(g)-\psi_{-}(g) = \psi(-1)\psi(-g)$
  for all $g\in   \Sp(V)$.
  Since $\psi(-1)=\pm 1$, the two claims 
  of the proposition follow from each other.
  
  Let $g \in \Sp(V)$ have odd order.
  Then $\C_{V}(-g) = \Ker(1+g) = 0$.
  It follows from \cref{p:absvalue} that
  $\abs{\psi_{+}(g)-\psi_{-}(g)}=\abs{\psi(-g)}=1$.
  Let $U\leq \Sp(V)$ be a subgroup of odd order
  and let $(\cdot,\cdot)_U$ denote the usual inner product
  for class functions on $U$.  
  Consider the virtual character 
  $\psi_{+}-\psi_{-}$.
  Then 
  \[
    (\psi_{+}-\psi_{-}, \psi_{+}-\psi_{-})_U 
      = \frac{1}{ \abs{U} } \sum_{g\in U} 
                        \abs{\psi_{+}(g)-\psi_{-}(g)}^2
      = 1
  \]
  and thus
  $\pm(\psi_{+}-\psi_{-})_U\in \Irr U$.
  Since $\psi_{+}(1)-\psi_{-}(1) = 1$,
  it follows that $\mu:=(\psi_{+}-\psi_{-})_U$
  is a linear character of $U$.
  Taking determinants yields
  $\det(\psi_{+})_U = \mu \det(\psi_{-})_U$
  and thus $\mu=1_U$ by the definition of 
  the canonical Weil character $\psi$.
  This shows
  \( 1 = (\psi_{+}-\psi_{-})(g) 
  \)
  for all $g$ of odd order,
  as claimed.
\end{proof}

The next result is the second part of
\cref{ti:valuesodd} from the introduction.

\begin{cor}\label{c:oddordervalues}
  Let $g\in \Sp(V)$ have odd order.
  Then
  \[ \psi(g) = \frac{1}{\sqrt{\abs{V}}}
               \sum_{v\in V} \lambda\left( \tfrac{1}{2}
                                \omega(v,vg) 
                                \right) .  
  \]
\end{cor}
\begin{proof}
    By the convolution formula from
    \cref{p:mult} applied to $g = (-1)(-g)$,
    we have
    \[ \psi(g) 
       = \frac{\psi(-1)\psi(-g)}{\sqrt{\abs{V}}}
         \sum_{ x\in V(1-(-1))\cap V(1-(-g)) }
            \lambda \left( \tfrac{1}{2} 
                           (B_{-1}(x,x) + B_{-g}(x,x) )
                    \right).
    \]
    By \cref{p:oddupsi},
    $\psi(-1)\psi(-g) = 1$.
    By \cref{l:symformqg} (or direct computation), 
    $B_{-1}(x,x) = 0$
    for all $x\in V(1-(-1))=2V=V$.
    As $g$ has odd order,
    $\Ker(1+g)=\{0\}$ and thus
    $V(1-(-g)) =V(1+g)=V$.
    For $x= v(1+g)$, we have
    \begin{align*}
    B_{-g}(x,x)
    &= \omega ( v , v(1+g) ) = \omega ( v ,vg ).
    \end{align*} 
    Thus the result follows.
\end{proof}

We conclude this section with a digression
and show that our definition of the 
\emph{canonical} Weil representation~$W$
is equivalent to the one of Isaacs \cite[(5.2)]{i73}.
This also yields a more representation theoretic
characterization of $W$.

\begin{remark}
   Let $\pi$ be the set of primes dividing $\abs{V}$,
   and let $G$ be any group acting on $V$ by symplectic
   automorphisms
   (so there is a homomorphism $G\to \Sp(V)$ defined).
   Then the restriction of the canonical Weil representation
   to $G$ 
   is the unique group homomorphism
   $W\colon G\to \cA^*$
   such that the following hold:
   \begin{enumerate}
   \item \label{it:autprop}
       $W(g)^{-1}a W(g) = a^g$ for all 
       $a \in \cA$, $g\in G$.
   \item \label{it:detord} 
       The order of the character $g\mapsto \det W(g)$
       is a $\pi$-number.
   \item \label{it:oddconst}
       For any $\pi$-subgroup $U$,
       the trivial character $1_U$ is the unique constituent
       of the character $\psi_U$ of $W_{|U}$ occurring with odd 
       multiplicity.
   \end{enumerate} 
\end{remark}
\begin{proof}
   The canonical Weil representation 
   has all these properties:
   \ref{it:detord} follows from
   \cref{p:rational}~\ref{it:detWeil}. 
   By \cref{p:oddupsi},
   we have $(\psi_{+} - \psi_{-})_U = 1_U$
   for any odd order subgroup $U$,
   and thus 
   $\psi_U = (\psi_{+}-\psi_{-})_U + 2(\psi_{-})_U 
                 = 1_U + 2(\psi_{-})_U $.
   Thus \ref{it:oddconst} holds for subgroups of odd order,
   and this includes $\pi$-subgroups.

   It remains to show that these properties determine $W$.
   Suppose $\widetilde{W}$ is another such
   homomorphism.
   Since $\Z(\cA)= \crp{K}1_{\cA}$,
    it follows from
   Condition~\ref{it:autprop} 
   that $\widetilde{W}(g) = \mu(g) W(g)$, where
   $\mu\colon G\to \crp{K}^*$ is a linear character.
   As $\det\widetilde{W}(g) = \mu(g)^{\psi(1)} \det W(g)$
   and $\det W(g)$ are supposed to have order a $\pi$-number,
   and since $\psi(1)=\sqrt{\abs{V}}$ is also a $\pi$-number,
   it follows that $\ord(\mu)$ is a $\pi$-number.
   Let $U$ be a $\pi$-subgroup. 
   As $1_U$ occurs in $\psi_U$ with odd multiplicity,
   $\mu_U$ occurs in $(\mu\psi)_U$ with odd multiplicity.
   By Condition~\ref{it:oddconst} for $\widetilde{W}$ 
   we must have 
   $\mu_U=1_U$. 
   Now we have shown that $\ord(\mu)$ is a $\pi$-number,
   but the restriction of $\mu$ to any $\pi$-subgroup
   is trivial. Thus $\mu=1$ and $\widetilde{W}=W$ as claimed.
\end{proof}

\section{Bilinear forms over principal ideal rings}
\label{sec:bil_pir}

This is the first of several sections which prepare the proof
of the main result (\cref{ti:main}).

Throughout this section,
$R$ denotes a \emph{principal ideal ring} (PIR),
that is, a commutative ring with~$1$
in which every ideal is a principal ideal.
We do \emph{not} assume that $R$ is a domain:
for example, the results hold for $R=\ints/m\ints$.
The ideas in the next proposition are extracted 
from an unpublished note by
P.-Y.~Gaillard~\cite{Gaillard11}.

\begin{prop}\label{p:pirbilin_id}
    Let $R$ be a PIR 
    and
    $B\colon U\times W\to R$ a bilinear form,
    where $U$, $W$ are $R$-modules.
    Then 
    $B(U,W):= \menge{B(u,w)}{u\in U,\, w\in W}$ 
    is an ideal,
    and if $B(U,W)=RB(u,w)$, then 
    \[ U= Ru + \leftidx{^B}{w}
    \quad \text{and}
    \quad
    W = Rw + u^B.
    \] 
    When $B$ is non-degenerate, these sums are direct sums.   
\end{prop}
\begin{proof}
    Let $\mathcal{I}$ be the set of all ideals of the 
    form $RB(x,y)$,
    and let $RB(u,w)$ be maximal in $\mathcal{I}$. 
    Since $B(U,w)$ is an ideal and $R$ a PIR, 
    we have $B(U,w)\in \mathcal{I}$,
    and so $B(U,w)=RB(u,w)$ by maximality.
    Let $x\in U$. Then $B(x,w)=rB(u,w)$ for some $r\in R$,
    and so $x-ru\in \leftidx{^B}{w}$.
    This shows $U = Ru + \leftidx{^B}{w}$.
    The same argument on the other side shows that
    $B(u,W)=RB(u,w)$ and
    $W= Rw + u^B$.
    When $B$ is non-degenerate, then
    $Ru \cap \leftidx{^B}{w} 
     \subseteq \leftidx{^B}{(Rw+u^B)} 
     = \leftidx{^B}{W}= \{0\}$
     and thus $U = Ru \oplus \leftidx{^B}{w}$.
     In the same way, $U^B=\{0\}$ implies 
     $Rw \cap u^B = 0$.
    
    It remains to show that
    $B(\leftidx{^B}{w}, u^B) \subseteq RB(u,w)$.
    Let $x\in \leftidx{^{B}}{w}{} $
    and $y\in u^{B}$.
    Then
    \( 
    B(ru+sx,w+y) =
    rB(u,w)+sB(x,y)
    \) for 
    \(r,s\in R \).
    It follows that the ideal
    \[ I=\menge{rB(u,w)+sB(x,y)}{r,s\in R} 
    \]
    is a member of $\mathcal{I}$. 
    By maximality, $I=RB(u,w)$, and thus
    $B(x,y)\in RB(u,w)$.
    Thus $B(U,W)= RB(u,w)$ is an ideal.
\end{proof}

\begin{cor}\label{c:pirbilin}
    Let $R$ be a PIR
    and
    $B\colon U\times W\to R$ a non-degenerate bilinear form,
    where the modules $U$ and $W$ are finitely generated.
    Then there are elements 
    $u_1$, $\dotsc$, $u_r\in U$
    and $w_1$, $\dotsc$, $w_r\in W$
    such that
    \[ U = Ru_1 \oplus \dotsb \oplus Ru_r
    \quad\text{and}\quad
    W = Rw_1 \oplus \dotsb \oplus Rw_r,\]
    and $B(u_i,w_j) = \delta_{ij} d_i$
    with $Rd_1 \geq Rd_2 \geq \dotsb \geq Rd_r$.
\end{cor}

\begin{proof}
    By \cref{p:pirbilin_id},
    $B(U,W) = RB(u_1,w_1)$
    for some $u_1\in U$ and $w_1\in W$,
    and we have 
    $U = Ru_1 \oplus \leftidx{^B}{w}{_1}$ and
    $W = Rw_1 \oplus u_1^B$.
    We can then repeat the argument for   
    $B\colon \leftidx{^{B}}{w}{_1} \times u_1^{B}\to R$.
    The process stops since $U$ and $W$ are noetherian.
\end{proof}

\begin{cor}\label{c:locpirsymform}
    Let $R$ be a local PIR such that $2$ is invertible in $R$.
    Let $Q\colon U\times U \to R$ be a symmetric bilinear
    form on the module $U$.
    Then there is $x\in U$ with
    $Q(U,U) = RQ(x,x)$.
\end{cor}
\begin{proof}
    Any ideal of $R$ has the form $R\pi^n$,
    where $R\pi$ is the maximal ideal of $R$.
    In particular, the ideals of $R$ form a chain.
    
    By \cref{p:pirbilin_id}, there exist 
    $u$, $w\in U$ with
    $Q(U,U)=RQ(u,w)$.
    As 
    \[ Q(u,w) = \frac{1}{2} 
    \big( Q(u+w,u+w) - Q(u,u) - Q(w,w)
    \big),
    \]
    and since the ideals of $R$ form a chain,
    we have $Q(U,U) = RQ(x,x)$ for
    at least one
    $x\in \{u+w,u,w\}$.
\end{proof}

\begin{cor}
\label{c:pirsymform}
    Let $R$ be a finite PIR of odd order
    and 
    $Q\colon U\times U \to R$ a symmetric bilinear
    form on the module $U$.
    Then there is $u\in U$ with
    $Q(U,U) = RQ(u,u)$.
\end{cor}
\begin{proof}
    Any finite ring is the direct product
    of finitely many finite, local rings,
    say $R= R_1\times \dotsm \times R_{\ell}$
    \cite[Theorem~8.7]{AtiyahMacdonald69}.
    We have a corresponding orthogonal 
    idempotent decomposition
    $1 = e_1 + \dotsb + e_{\ell}$,
    where $R_i=Re_i$ and $e_i= 1_{R_i}$.
    Then $U = e_1U\oplus \dotsb \oplus e_{\ell}U$ 
    and $Q(e_iU,e_jU)=0$ for $i\neq j$,
    and $Q(e_iU,e_iU)\subseteq Re_i=R_i$.
    By \cref{c:locpirsymform}, for each $i$,
    there is an element $u_i=e_iu_i\in e_iU$
    such that $Q(e_iU,e_iU)= R Q(u_i,u_i)$.
    Then $u=u_1+\dotsb + u_{\ell}$
    has the desired property.
\end{proof}

By the usual Gram-Schmidt process,
it follows that for $U$ finitely generated,
$U/U^Q$ can be written
as an orthogonal sum of cyclic modules.

As the proof shows,  the last corollary holds
when $R$ is a direct product of finitely many
local PIRs with $2$ invertible.
On the other hand, \cref{c:pirsymform} does not hold for
arbitrary PIRs,
not even when $2$ is invertible.
For example, 
the form over the polynomial ring $R=\GF{3}[x]$ 
on $U=R^2$ with Gram matrix
$\begin{psmallmatrix}
x & 1 \\ 1 & x-1
\end{psmallmatrix}
$
can not be written as orthogonal sum
\cite[Example~6.19(ii)]{Gerstein08}.

For symplectic forms, we have the following:

\begin{cor}\label{c:piralter}
    Let $R$ be a PIR and
    $\omega\colon V\times V \to R$ 
    an alternating and
    non-degenerate form on the 
    finitely generated $R$-module $V$.
    Then 
    $V= Re_1 \oplus \dotsb \oplus Re_m \oplus  
    Rf_1 \oplus \dotsb \oplus Rf_m$,
    where $\omega(e_i,f_j) = \delta_{ij}d_i$ 
    and $\omega(e_i,e_j)=\omega(f_i,f_j)=0$ for all $ i$, $j$,
    and $Rd_1 \geq Rd_2 \geq \dotsb \geq Rd_m$.
\end{cor}
\begin{proof}
    As in the proof of \cref{c:pirbilin}, begin with 
    $u_1=e_1$ and $w_1=f_1$ such that 
    $\omega(V,V) = R\omega(e_1,f_1)$.
    Then by \cref{p:pirbilin_id},
    $ V = Re_1 \oplus \leftidx{^{\perp}}{f_1}
    = Rf_1 \oplus e_1^{\perp}$.
    As $\omega$ is alternating, we have $e_1\in e_1^{\perp}$
    and $f_1\in f_1^{\perp} = \leftidx{^{\perp}}{f_1}$.
    So in the next step, we can
    choose $u_2=f_1$ and $w_2= e_1$. The proof follows.
\end{proof}

\section{Signs of automorphisms}
\label{sec:sign}

Let $U$ be a finite abelian group.
Every automorphism $g$ of $U$ permutes $U$,
and we write $\sign(g)=\sign_U(g)$ 
for the sign of this permutation.
Thus we have a natural character
$\sign\colon \Aut(U) \to \{ \pm 1\}$.
This character has been studied by a number of people,
in particular P.~Cartier~\cite{Cartier70}
and A.~Brunyate and P.~L.~Clark~\cite{BrunyateClark15}.
We state 
and prove  
the results we need later. 

In the first two results, we actually do not need that 
$U$ is abelian, and so we use multiplicative notation
(so $-u$ becomes $u^{-1}$, and so on).
These results are due to Cartier~\cite[p.~38--39]{Cartier70}.

\begin{lemma}
\label{l:gausszolotarev}
  Let $\pi$ be a permutation of a finite group $U$
  which commutes with taking inverses:
  for all $u\in U$, we have $(u^{-1})\pi = (u\pi)^{-1}$.
  Choose $P\subseteq U$ such that
  $U = P \dcup P^{-1} \dcup I$ (disjoint union),
  where $I=\menge{ u\in U }{ u^2 = 1 }$.
  Then
  \[ \sign(\pi)
     = (-1)^{\abs{P\pi\cap P^{-1}}} \sign(\pi_I).
  \]
\end{lemma}
\begin{proof}
  The assumption on $\pi$ yields that
  $\pi$ maps $I$ onto itself.
  Let $\tau$ be the product of all the transpositions
  $(u,u^{-1})$ with $u\in P\pi\cap P^{-1}$.
  Then $\pi\tau$ maps $P$, $P^{-1}$ and $I$ onto itself, 
  and the permutations on $P$ and $P^{-1}$ are related  
  by $(u^{-1})(\pi\tau) = (u\pi\tau)^{-1}$.
  Thus the restriction of $\pi\tau$ to
  $P\dcup P^{-1}$ is an even permutation.
  The result follows.  
\end{proof}

\begin{lemma}\label{l:signreduc}
  Let $U$ be a finite group of odd order
  and $N\nteq U$  a normal subgroup.
  Let $\alpha\in \Aut(U)$ be an automorphism that
  maps $N$ into itself.
  Then $\sign(\alpha) = \sign(\alpha_{U/N}) \sign(\alpha_N)$.  
\end{lemma}
\begin{proof}
  Choose $P_1\subseteq N$ and $P_2\subseteq U/N$ such that
  $N= P_1 \dcup P_1^{-1} \dcup\{1 \}$ and
  $U/N = P_2 \dcup P_2^{-1} \dcup \{1_{U/N}\}$.
  Let $\hat{P_2} = \menge{ u\in U }{ Nu \in P_2 }$
  be the pre-image of $P_2$ in $U$.
  Then for $P= \hat{P_2} \cup P_1$, 
  we have $U=P\dcup P^{-1} \dcup \{1\}$.
  By \cref{l:gausszolotarev}, we have
  \begin{align*}
    \sign(\alpha) = (-1)^{\abs{P\alpha\cap P^{-1}}}
      &= (-1)^{ \abs{ \hat{P_2}\alpha \cap \hat{P_2}^{-1}} }
        (-1)^{ \abs{ P_1\alpha \cap P_1^{-1}} }
      \\
      &= (-1)^{ \abs{N} \abs{ P_2\alpha_{U/N} \cap P_2^{-1} } }
        \sign(\alpha_N)
      \\ &
      = (-1)^{ \abs{ P_2\alpha_{U/N} \cap P_2^{-1} } }
        \sign( \alpha_N )
      \\&
      = \sign( \alpha_{U/N} ) \sign( \alpha_N ).
  \qedhere
  \end{align*}
\end{proof}

Let $R$ be a finite ring and 
$a\in R$ be invertible.
We write $\sign_R(a)$ for the sign of the permutation of $R$
defined by
$r\mapsto ra$.

\begin{lemma}[Zolotarev, Lerch, Frobenius] 
\label{l:zolofrob}\hfill
  \begin{enumerate}
  \item Let $\crp{F}$ be a finite field of odd order
        and $0\neq a\in \crp{F}$.
        Then $\sign_{\crp{F}}(a) = 1$ if and only if
        $a$ is a square in $\crp{F}$.
  \item Let $R=\ints / m$ with $m$ odd and
          $a\in R^*$.
        Then $\sign_R(a) = \Jacobi*{a}{m}$
        (the Jacobi symbol).
  \end{enumerate}
\end{lemma}
\begin{proof}
  $\crp{F}^*$ is a cyclic group of even order, 
  and the squares form the unique subgroup of index $2$.
  When $\crp{F}^* =\erz{a}$, then 
  $\sign_{\crp{F}}(a) = -1$, because 
  the corresponding permutation forms one long
  cycle of length $\abs{\crp{F}^*} = \abs{\crp{F}}-1$.
  This shows the first part.
  
  The second part can be proved using \cref{l:signreduc}
  and induction on $m$:
  For $m=k\ell$ with $k$, $\ell >1$,
  we have $R/kR \iso \ints / k$ and 
  $kR \iso \ints/ \ell$.
  For $m$ prime, the result follows from the first part.
\end{proof}

We see that \cref{l:gausszolotarev} generalizes
the Gauss-Schering lemma from elementary number theory.
The next result was proved by I.~Schur~\cite[p.~151]{Schur21} for
$R=\ints/m\ints$,
by P.~Cartier~\cite[p.~41]{Cartier70}
for finite fields,
and by Brunyate and Clark~\cite[Theorem~6.1]{BrunyateClark15}
as part of a more general result.

\begin{lemma}\label{l:detsign}
  Let $R$ be a commutative ring of finite, odd order,
  and let $g\in \Aut_R(R^d)= \GL(d,R)$.
  Then 
  \[\sign_{R^d}(g) = \sign_R( \det(g)).
  \]
  (On the left, we view $g$ as a permutation of $R^d$,
  and on the right, $\det(g)$ as a permutation of $R$.)
\end{lemma}
\begin{proof}
  Suppose that $R = R_1 \times R_2$, a direct product of two
  rings.
  Let $1=e_1+e_2$ be the corresponding idempotent decomposition.
  Any $R$-module $U$ is the direct sum
  $U=Ue_1 \oplus Ue_2$.
  We see that $\det(g) = \det(ge_1) + \det(ge_2)$,
  where $\det(ge_i) \in R_i=Re_i$ is the determinant
  of $ge_i$ as element in $\GL(d,R_i)$.
  Also, $\det(g)$ acts on $R_i$ in the same
  way as $\det(ge_i)=\det(g)e_i$,
  and $g$ acts on $R_i^d = R^de_i$ as $ge_i$ does.
  By \cref{l:signreduc},
  we are reduced to prove the lemma 
  for $R_1$ and $R_2$.
  
  Since any finite ring $R$ is the direct product
  of finitely many \emph{local} rings
  \cite[Theorem~8.7]{AtiyahMacdonald69},
  we may assume that $R$ is local.
  Then any column of an invertible matrix over $R$
  contains at least one unit.
  By the usual Gauss elimination process,
  it follows that an invertible matrix is a product
  of matrices which differ in exactly one entry from
  the identity matrix
  (when this entry is on the main diagonal,
  it must be a unit of $R$).
  For matrices of this kind, the assertion follows
  immediately from \cref{l:signreduc}.
\end{proof}

For the rest of this section, 
we assume again that
$R$ is a finite, commutative ring
with a primitive additive
 character $\lambda\colon R\to \compl^*$,
in which $2$ is invertible.

\begin{defi}
   Let $B$, $C\colon U \times W\to R$ be 
   two non-degenerate bilinear forms
   on the finite $R$-modules $U$ and $W$.
   Then there is a unique $\alpha\in \Aut_R(U)$
   such that $B(u,w) = C(u\alpha,w)$ for all $u\in U$,
   $w\in W$.
   We define $\sign(B/C):= \sign(\alpha)$.
\end{defi}
The existence of $\alpha$ in the preceding definition 
follows since $B$ and $C$ both induce isomorphisms
$U\to  \Hom(W,R)$ 
(\cref{l:perp_basic}).

We will only need the case $U=W$.
The definition is maybe motivated by the following lemma
and its proof:

\begin{lemma}
 \label{l:signdisc}
    Suppose $R$ is a finite field
    and $B$, $C\colon U\times U\to R$ two bilinear forms.
    Then $\sign(B/C)=1$ if and only if
    $\disc(B) = \disc(C)$.
\end{lemma}
\begin{proof}
        Let $S$ be the standard inner product with 
    respect to some basis of $U$,
    and let $G_B$ and $G_C$ be the Gram matrices of 
    $B$ and $C$ 
    with respect to that basis.
    If we view $G_B$ as the matrix of a linear map,
    then obviously 
    $B(u,w)= S(uG_B,w)$ for all $u$, $w\in U$.
    Similarly, $C(u,w)=S(uG_C,w)$.
    Thus $B(u,w)= C(uG_B (G_C)^{-1},w)$
    and so 
    $\sign(C/B) = \sign (G_BG_C^{-1}) 
    = \sign_{R}(\det G_B) \sign_R (\det G_C)^{-1}$
    by \cref{l:detsign}.
    By \cref{l:zolofrob}, $\sign_{R}$ is the quadratic
    character on $R$.
    By definition,
    the discriminant of $B$
    is $\det G_B$ modulo the squares in $R$.
    The result follows.
\end{proof}

\begin{lemma}\label{l:disc_dirprod}
    Let $B\colon U\times U\to R$ be a non-degenerate 
    bilinear form on the finite module $U$.
    Suppose that $U = X \oplus Y$ with
    $B(Y,X) = 0$.
    When $Q_X$ and $Q_Y$ are non-degenerate 
    forms on $X$ and $Y$, respectively,
    then 
    $\sign((Q_X\oplus Q_Y)/B) 
      = \sign(Q_X/B_{X}) \sign(Q_Y/B_{|Y})$.
\end{lemma}
Of course, 
$Q_X \oplus Q_Y$ is the form on $X\oplus Y$ defined
by
\[  (Q_X\oplus Q_Y)(x_1+y_1,x_2+y_2)
= Q_X(x_1,x_2) + Q_Y(y_1,y_2)
\]
for $x_1$, $x_2\in X$,
$y_1$, $y_2\in Y$. 
When $Q_X$ and $Q_Y$ are symmetric, then so is
$Q_X \oplus Q_Y$.
\begin{proof}[Proof of \cref{l:disc_dirprod}]
    As $B(Y,X)=0$ and $B$ is non-degenerate, 
    the restrictions
    $B_{|X}$ and $B_{|Y}$ are also non-degenerate.
    Thus there are $\sigma\in \Aut(X)$ and $\tau\in \Aut(Y)$
    with
    \[ Q_X(x_1,x_2) = B(x_1\sigma,x_2)
    \quad \text{and} \quad
    Q_Y(y_1,y_2) = B(y_1\tau,y_2). 
    \]
    Since $B_{|Y}$ is non-degenerate,
    there is, for each $x\in X$,
    an element
    $x\kappa \in Y$ such that $B(x,y)=B(x\kappa,y)$
    for all $y\in Y$.
    The map $\kappa\colon X\to Y$ is an homomorphism.
    Define $\alpha\colon U\to U$ by
    $(x+y)\alpha = x\sigma - x\sigma \kappa + y\tau$.
    Then
    \begin{align*}
    B((x_1+y_1)\alpha, x_2+y_2)
    &= B( x_1\sigma - x_1\sigma\kappa + y_1\tau, x_2 + y_2 )
    \\ &= B( x_1\sigma ,x_2 ) + B( x_1\sigma, y_2 )             
    - B( x_1\sigma\kappa, y_2 )
    + B( y_1\tau, y_2 )
    \\ &
    = Q_X(x_1,x_2) + Q_Y(y_1,y_2)
    \\ & = (Q_X \oplus Q_Y) (x_1+y_1,x_2+y_2).
    \end{align*}
    Thus 
    $\sign((Q_X\oplus Q_Y)/B) = \sign(\alpha)$.
    Now $\alpha_{|Y} = \tau$,
    and $\alpha$ on $U/Y\iso X$ is $\sigma$.
    The result follows from \cref{l:signreduc}.
\end{proof}

\section{Quadratic Gauss sums on abelian groups}
\label{sec:gausssums}
As in \cref{basicsetup},
let $R$ be a finite commutative ring of odd order
(equivalently, $2$ is invertible in $R$), and let
$\lambda\colon R\to \compl^*$ be a linear character
such that $\Ker \lambda$ contains no nonzero ideal of $R$.
Let $X$ be a finite $R$-module, 
and let $q\colon X\times X\to R$ be a 
symmetric, non-degenerate bilinear form.
We define the (generalized quadratic) \emph{Gauss sum} 
$\gamma_{\lambda}(q)$
corresponding to $\lambda$ and $(X,q)$ as
  \[ \gamma_{\lambda}(q) := 
     \gamma_{\lambda}(X,q) := \frac{1}{ \sqrt{ \abs{X} } }
                     \sum_{x\in X} \lambda(\tfrac{1}{2} q(x,x)).
  \]
(The factor $\tfrac{1}{2}$ is there to obtain consistency
with the \emph{Weil index} over finite fields 
\cite{thomas08,Weil64}.
Otherwise, this factor is not really important.)

By \cref{l:perp_basic},
the form
$\lambda\circ q\colon X\times X \to \compl^*$ 
is also non-degenerate,
and obviously,  
$\gamma_{\lambda}(q)$ depends only on $\lambda\circ q$.

Gauss sums have the following well-known properties:
\begin{lemma}\label{l:gauss_elem}
\leavevmode
   \begin{enumerate}
   \item \label{it:gaussmult}
        $\gamma_{\lambda}(q_1 \oplus q_2)
         = \gamma_{\lambda}(q_1) \gamma_{\lambda}(q_2)$
         for forms $q_i$ on $X_i$ ($i=1$, $2$)
         and $q_1\oplus q_2$ their direct sum,
         a form on $X_1\oplus X_2$.
   \item \label{it:gaussred}
       When $U\leq X$ is isotropic 
     (that is, $U\subseteq U^q$),
     then $q$ induces a non-degenerate form $\widetilde{q}$
     on $U^q/U$, and
     $\gamma_{\lambda}(q) = \gamma_{\lambda}(\widetilde{q})$.
   \item \label{it:lagrgauss}
         When $(X,q)$ contains a Lagrangian submodule $L$
         (that is, $L^{q}=L$), 
         then $\gamma_{\lambda}(q) = 1$.
   \item \label{it:absgauss} 
         $\abs{\gamma_{\lambda}(q)} = 1$.
   \end{enumerate}
\end{lemma}
\begin{proof}
   Statement~\ref{it:gaussmult} is a routine computation.
   In the situation of~\ref{it:gaussred},
   we have $(U^q)^q=U$,
   so $\widetilde{q}(s_1+U,s_2+U):= q(s_1,s_2)$
   (where $s_1$, $s_2\in U^q$)
   is well-defined and non-degenerate.
   Write
   \[   X = \bigdcup_{t\in T } (t+ U^q)
      \quad\text{and}\quad
        U^q = \bigdcup_{s\in S} (s+U).
   \]
   We assume that $0\in T$. 
   For $t\in T$, $s\in S$ and $u\in U$,
   \[q(t+s+u,t+s+u) = q(t,t) + q(s,s) + 2q(t,s) + 2q(t,u) ,
   \]
   as $q(u,u)=q(s,u)=0$. 
   Thus
   \begin{align*}
      \gamma_{\lambda}(q) 
        &= \frac{1}{\sqrt{\abs{X}}} 
           \sum_{t\in T} \sum_{s\in S} \sum_{u\in U} 
             \lambda \big( \tfrac{1}{2}q(t,t) + \tfrac{1}{2}q(s,s)
                           + q(t,s) +q(t,u)
                    \big)
     \\ &= \frac{1}{\sqrt{\abs{X}}} 
           \sum_{t\in T} 
             \lambda \big( \tfrac{1}{2}q(t,t) \big)
           \sum_{s\in S} 
             \lambda \big( \tfrac{1}{2}q(s,s) \big) \, 
             \lambda \big( q(t,s) \big) \, 
           \sum_{u\in U} 
             \lambda \big( q(t,u) \big)         
     \\ &= \frac{\abs{U}}{\sqrt{\abs{X}}} 
           \sum_{s\in S} 
            \lambda \big( \tfrac{1}{2}q(s,s) \big) 
         = \gamma_{\lambda}(\widetilde{q})\,.
   \end{align*}
   Here the third equality follows from
   \( \sum_{u\in U} \lambda(q(t,u)) = 0 
   \)
   unless $t \in U^{q}$,
   in which case the sum is $\abs{U}$,
   and $t=0$ by our assumption.   
   The last equality follows from 
   $\abs{U}\abs{U^q}=\abs{X}$ (\cref{l:perp_basic}).
   This shows \ref{it:gaussred}, and \ref{it:lagrgauss}
   is a special case.
   
   We have
   \[ \abs{\gamma_{\lambda}(q)}^2 = 
      \gamma_{\lambda}(q) \cconj{\gamma_{\lambda}(q)}
      = \gamma_{\lambda}(q) \gamma_{\lambda}(-q)
      = \gamma_{\lambda}(q \oplus(-q)).
   \]
   But $(X\oplus X,q\oplus(-q))$ has a Lagrangian submodule,
   namely 
   $L= \menge{(x,x)}{ x\in X }$.
   Thus \ref{it:absgauss} follows from \ref{it:lagrgauss}.
\end{proof}

We note in passing that \ref{it:gaussmult} and \ref{it:gaussred}
reduce the computation of $\gamma_{\lambda}$ to the case
where $(X,q)$ is anisotropic and indecomposable.
In this case, $F:=R/\ann_R(X)$ is a field and $\dim_F (X) \leq 1$.

Let $q\colon X\times X\to R$ be symmetric and non-degenerate
as before. 
Following I.~Schur~\cite{Schur21}, cf.~\cite{MurtyPathak17},
we consider the $X\times X$-matrix
\[ F_{\lambda}(q) = 
   \frac{1}{\sqrt{\abs{X}}} 
   \big( \lambda( \tfrac{1}{2} q(x,y)
                ) 
   \big)_{x,y\in X}.
\]
Obviously, $\tr(F_{\lambda}(q)) = \gamma_{\lambda}(q)$.
Our proofs of the next results are  straightforward generalizations 
of Schur's arguments. 

\begin{prop}\hfill 
\label{p:gaussiansigns}
  \begin{enumerate}
  \item \label{it:gausssquare}
        $\displaystyle 
         \gamma_{\lambda}(q)^2 = (-1)^{ \left( \frac{\abs{X}-1}{2}
                                \right)
                              }
                       = \Jacobi*{-1}{\abs{X}}$.
  \item \label{it:schur21}
        $\displaystyle 
         \gamma_{\lambda}(q) = (-1)^{ \left( \frac{\abs{X}^2 - 1}{8}
                              \right) 
                            }  \det(F_{\lambda}(q))
                     = \Jacobi*{2}{\abs{X}} \det(F_{\lambda}(q))$.
  \end{enumerate}
  (Again, $\Jacobi{\cdot}{n}$ denotes the Jacobi symbol.)  
\end{prop}
\begin{proof}
    Write $F:= F_{\lambda}(q)$ and $T=F^2$.
    Then the entry $t_{x,y}$ of $T$ is 
    \begin{align*}
      t_{x,y} &= \frac{1}{\abs{X}} \sum_{u\in X} 
                  \lambda(\tfrac{1}{2}q(x,u)) 
                  \lambda(\tfrac{1}{2}q(u,y))
        \\    &= \frac{1}{\abs{X}} \sum_{u \in X} 
                   \lambda( \tfrac{1}{2}q(x+y,u) )
        = \delta_{x+y,0}.
    \end{align*}   
   Since $\abs{X}$ is odd, we see that we can arrange the elements
   of $X$ such that $F^2$ has the form
   \[ F^2 = \begin{pmatrix}
              1 & 0 & 0 \\
              0 & 0 & I \\
              0 & I & 0
            \end{pmatrix}.
   \]            
   It follows that $F^4 = I$.
   
   It follows that the eigenvalues of $F$
   are from the set $\{ \pm 1, \pm i\}$.
   Let $m_k$ be the multiplicity of the eigenvalue~$i^k$
   ($k=0$, $1$, $2$, $3$).   
   Thus
   \[ G:= \gamma_{\lambda}(q) = (m_0-m_2) + (m_1-m_3)i
      \quad\text{and}\quad
      \det(F) = i^{2m_2 + m_1 - m_3}.
   \]
   From $\abs{G}=1$ we conclude that
   $G\in \{\pm 1, \pm i \}$.
   By looking at the traces of $F^k$ ($k=0$, $1$, $2$, $3$),
   we get the four equalities
   \begin{align*}
     \sum m_k &= \abs{X}, & \sum m_k i^k &= G, \\
     \sum m_k (-1)^k &= 1, & \sum m_k i^{-k} &= \cconj{G}.
   \end{align*}
   (Recall that we have computed $F^2$ above.)
   Thus
   \[ 4m_0 = \abs{X} + 1 + G + \cconj{G}
      \quad\text{and}\quad
      4m_2 = \abs{X} + 1 - (G+\cconj{G}).
   \]
   Since the right hand sides must be divisible by $4$
   and $G\in \{ \pm 1, \pm i \}$, it follows that
   $G\in \{\pm 1\} $ when $\abs{X}\equiv 1 \mod 4$
   and $G\in \{ \pm i \} $ when
   $\abs{X} \equiv -1 \mod 4$.
   This yields \ref{it:gausssquare}.

   To see \ref{it:schur21},
   assume first that $\abs{X} \equiv 1 \mod 4$,
   so $G= \pm 1$.
   Then we have 
   \[ m_2 = \frac{\abs{X}+1- 2 G}{4}
      \quad\text{and}\quad
      \det(F) = (-1)^{m_2}.
   \]
   It follows that 
   $G = \det(F)$ when $\abs{X}\equiv 1 \mod 8$,
   and $G=-\det(F)$ when $\abs{X} \equiv 5 \mod 8$.
   Thus \ref{it:schur21} holds for $\abs{X}\equiv 1 \mod 4$.
   
   Now assume $\abs{X}\equiv -1 \mod 4$,
   so $G = (m_1-m_3)i = \pm i$. 
   In this case we have
   \[ \det(F) = (-1)^{m_2} i^{m_1-m_3}
      =(-1)^{m_2} (m_1-m_3)i
      =(-1)^{m_2} G.
   \]
   Since now
   \[ m_2 = \frac{\abs{X}+1}{4},
   \]
   we get that $m_2$ is even when $\abs{X}\equiv -1 \mod 8$,
   and $m_2$ is odd when $\abs{X} \equiv 3 \mod 8$,
   and \ref{it:schur21} holds also in this case.
\end{proof}

\begin{cor}\label{c:gauss_ae}
  Suppose the symmetric, non-degenerate forms 
  $q_1$, $q_2\colon X\times X \to R$ are related by
  $q_2(x,w) = q_1(x\sigma, w)$, where $\sigma\colon X\to X$.
  Then $\gamma_{\lambda}(q_2) = \sign(\sigma) \gamma_{\lambda}(q_1)$.
\end{cor}

In the case where $X=(\ints/m\ints)^d$,
this result
dates back to H. Weber~\cite{Weber1872}
(cf. C.~Jordan~\cite{Jordan1871}).
Our proof is essentially the same as
Schur's proof~\cite{Schur21}.

\begin{proof}
  Notice that $\sigma$ is necessarily invertible and thus
  induces a permutation of $X$.
  We have $F_{\lambda}(q_2)= P(\sigma) F_{\lambda}(q_1)$,
  where $P(\sigma)$ is the permutation matrix corresponding
  to $\sigma$.
  Thus $\det(F_{\lambda}(q_2)) = \sign(\sigma)\det(F_{\lambda}(q_1))$.
  By \cref{p:gaussiansigns}\ref{it:schur21}, 
  the quotient $\gamma_{\lambda}(q_i)/\det(F_{\lambda}(q_i))$
  depends only on $\abs{X}$, but not on the form~$q_i$ itself.
  The result follows.
\end{proof}

\begin{remark}
    We should mention here that the matrix $F_{\lambda}(q)$ 
    can be interpreted as the image of a certain $g\in \Sp(V)$
    under an explicit matrix version of the Weil representation
    (up to a scalar).
    Namely, let $V = X\oplus X$ with symplectic form
    $\omega( (x,y), (z,w) ) = \tfrac{1}{2}(q(x,w)-q(y,z))$,
    and let $g\in \Sp(V)$ be 
    defined by $(x,y)g=(-y,x)$.
    
    Let $L= X\oplus 0$ and $T = 0 \oplus X$
    (Lagrangian submodules).
    In the proof of \cref{p:sympl_alg_matring}, we
    constructed an explicit isomorphism between 
    the symplectic algebra~$\cA$
    and $\mat_{X}(\crp{K})$ associated to $L$ and $T$.    
    By a (tedious) calculation,
    one can show that under this isomorphism,
    $P(g)$ corresponds to 
    $\abs{X}\gamma_{\lambda}(-q) F_{\lambda}(q)$,
    where $P(g)$ is as in \cref{p:weilconcret}.
    
    It is also well known that for
    $X=R=\ints/m$ and $q(a,b)=ab$, 
    the matrix $F_{\lambda}(q)$ encodes the discrete 
    Fourier transform~\cite{AuslanderTolimieri79,GurevichHadani09}.
    (Usually, it is defined without the factor~$\tfrac{1}{2}$,
    so that the case of even $m$ is also covered.)
\end{remark}

\section{Factorization of a symplectic automorphism}
\label{sec:factorization}

Assume \cref{basicsetup}.
(The results in this section can be extended 
to more general situations, 
but to keep the notation simple,
we do not assume this greater generality.)
Recall that in \cref{l:sigmaform}, we defined a 
non-degenerate bilinear form 
$B_g\colon V(1-g)\times V(1-g) \to R$
for any $g\in \Sp(V)$, and that this form 
has the property
\begin{equation*}
B_g( x, y ) - B_g( y, x ) = \omega( x, y )
\quad\text{for all } x,y\in V(1-g).
\end{equation*}

\begin{prop}\cite[Theorem~1.1.1--2]{Wall63} \cite[Theorem~3.5]{Ward17}
    \label{p:spparam}
    Let $X$ be a submodule of $V$ and    
    $B\colon X\times X\to R$ 
    a non-degenerate bilinear form  with 
    \begin{equation}
    \label{equ:B_alt} 
    B( x, y ) - B( y, x ) = \omega( x, y )
    \quad\text{for all } x,y\in X.
    \end{equation}
    Then there exists a unique $g\in \Sp(V)$ such that
    $X=V(1-g)$ and $B = B_g$.
\end{prop}
\begin{proof}
    Define $\alpha\colon V\to X$ by 
    requiring $\omega ( v, x ) = B( v\alpha, x )$
    for all $x\in X$ and $v\in V$.
    This is possible since the non-degenerate form
    $B$ induces an isomorphism
    from $X$ to $\Hom_R(X,R)$
    (\cref{l:perp_basic}). 
    Then $\alpha$ is $R$-linear with
    $\Ker \alpha = X^{\perp}$ and $V\alpha \subseteq X$,
    so 
    $V\alpha=X$. 
    
    Set $g= 1_V -\alpha$. 
    This is the unique map $g\colon V\to V$ with
    $\omega(v,x)=B(v(1-g),x)$ for all $v\in V$ and $x\in X$.
    It remains to show that $g$ preserves the form $\omega$:
    \begin{align*}
    \omega( vg, wg )
    &= \omega( v, w ) - \omega( v, w\alpha )
    - \omega( v\alpha, w )
    + \omega( v\alpha, w\alpha )
    \\ &= \omega( v, w ) - B( v\alpha, w\alpha )
    + B( w\alpha, v\alpha )
    + \omega( v\alpha, w\alpha )
    \\ &= \omega( v, w ),
    \end{align*}
    where the last equality follows from \eqref{equ:B_alt}.
    Thus $g\in \Sp(V)$ as claimed.
\end{proof}

\begin{lemma}
    Let $h$, $k\in \Sp(V)$ and assume that
    $V(1-h)\cap V(1-k)=\{0\}$.
    Then
    \[ V(1-hk) = V(1-h) \oplus V(1-k)
    \quad\text{and} \quad
    B_{hk}(V(1-k),V(1-h)) = 0.
    \]
\end{lemma}
\begin{proof}
    As $v(1-hk) = v(1-h) + vh(1-k)$,
    we always have
    $  V(1-hk) \subseteq V(1-h) + V(1-k) $.
    From  $V(1-h)\cap V(1-k)=\{0\}$ it follows that
    \begin{align*}
    V = \{0\}^{\perp} 
    &= (V(1-h)\cap V(1-k) )^{\perp}
    \\ &= (V(1-h))^{\perp} + (V(1-k))^{\perp}  
    &&\text{(\cref{l:perp_basic}\ref{it:perp_sum})}
    \\ &= \C_V(h) + \C_V(k),
    &&\text{(\cref{l:ker_senkr})}
    \end{align*}
    and thus also
    \[ V = Vh^{-1} = \C_V(h) + \C_V(k) h^{-1}.
    \]
    It follows that 
    \begin{align*}
    V(1-h) &= \C_V(k)(1-h) = \C_V(k)h^{-1}(1-h) 
    \\
    \text{and}\quad
    V(1-k) &= \C_V(h)(1-k).
    \end{align*}
    But for $d\in \C_V(k)h^{-1}$, we have
    $ d(1-h) = d - dh = d - dhk = d(1-hk) \in V(1-hk)$,
    which shows $V(1-h) \subseteq V(1-hk)$.
    Similarly,
    $ c ( 1-k ) = c - ck = c - chk = c (1-hk) $ 
    for $c\in C_V(h)$, 
    and so 
    $ V(1-k) \subseteq V(1-hk) $. 
    
    For $x\in V(1-h) $ and $y=c(1-k)\in V(1-k)$
    with $c\in \C_V(h)$
    we have
    \begin{align*}
    B_{hk}( c(1-k), x )
    = B_{hk}( c(1-hk), x )
    = \omega ( c, x ) = 0
    \end{align*}
    as $\C_V(h) \perp V(1-h)$.
\end{proof}

\begin{prop}
\label{p:factorize}
    Let $g\in \Sp(V)$ and suppose that
    $V(1-g) = X \oplus Y$ with
    $B_g(Y,X) = 0$.
    Then there exist $h$, $k\in \Sp(V)$ with
    $X=V(1-h)$, $Y=V(1-k)$
    and $B_h = (B_g)_{|X}$,
    $B_k = (B_g)_{|Y}$.
    For this $h$ and $k$, we have $g=hk$.
\end{prop}
\begin{proof}
    As $B_g(Y,X)=0$, the restrictions
    $(B_g)_{|X}$ and
    $(B_g)_{|Y}$ are nondegenerate.
    By  \cref{p:spparam},
    there exist $h$ and $k\in \Sp(V)$ such that
    $V(1-h)=X$ and $V(1-k)=Y$,
    and
    $\omega ( v, x ) = B_g( v(1-h), x )$
    and $\omega ( v, y ) = B_g( v(1-k), y )$
    for all $x\in X$, $y\in Y$ and $v\in V$.
    
    Let $\alpha = 1-h $ and $\beta = 1-k $.
    We want to show $g=hk$,
    which is equivalent to
    $1-g = \alpha+\beta-\alpha\beta$.
    Let $x\in X$, $v\in V$.
    Using $V\beta =  Y$
    and $B_g(Y,X)= 0$,
    we see that 
    \begin{align*}
    B_g( v ( \alpha + \beta - \alpha\beta ), x )
    = B_g( v\alpha, x ) 
    &= \omega ( v, x ) 
    = B_g( v(1-g), x ) \,.
    \end{align*}
    
    Next, let $y\in Y$, $v\in V$.
    Then
    \begin{align*}
    B_g( v ( \alpha + \beta - \alpha\beta), y )
    &= B_g( v\alpha, y ) 
    + B_g( v ( 1 - \alpha ) \beta, y )
    \\ &= B_g( v\alpha, y ) + \omega ( v ( 1-\alpha ), y ).
    \intertext{By \cref{l:sigmaform}, we have
        $B_g(v\alpha,y) - B_g(y, v\alpha) 
        = \omega ( v\alpha, y )$.
        Together with $B_g(y,v\alpha)\in B_g(Y,X)=0$, we get%
    } 
    B_g(v(\alpha+\beta+\alpha\beta),y)
    &= \omega( v\alpha,y ) 
                         + \omega( v(1-\alpha),y )
    \\ &= \omega ( v, y )
    = B_g( v(1-g), y ).
    \end{align*}
    We have now shown that
    \[ B_g( v(\alpha+\beta-\alpha\beta), z )
    = B_g( v(1-g), z )
    \]
    for all $z\in X \cup Y$ and $v\in V$.
    As $B_g$ is nondegenerate on
    $V(1-g) = X\oplus Y$, it follows that
    $1-g = \alpha+\beta-\alpha\beta$ and thus 
    $g=hk$.
\end{proof}

\section{Proofs of the main theorems}
\label{sec:values}

This section is devoted to the proofs of 
\cref{ti:main} and \cref{ti:inv_value}.
Throughout, we assume \cref{basicsetup}.
We will need the following formulation of 
\cref{c:oddordervalues}
(\cref{ti:valuesodd}):

\begin{cor}
\label{c:oddorderQg}
     When $g\in \Sp(V)$ has odd order, then
    \[ \psi(g) = \sqrt{\abs{\C_V(g)}} \gamma_{\lambda}(-Q_g),
    \]
    where 
    $Q_g(x,y) = (1/2)\big(B_g(x,y)+B_g(y,x) \big) = B_g(x\tfrac{1+g}{2},y)$
    as in \cref{l:symformqg}.
\end{cor}
\begin{proof}
    Since
    $\omega(v,vg) = - \omega(v,v(1-g)) = -B_g(x,x) = -Q_g(x,x)$
    for $x=v(1-g)$,
    this is just a rewording of the formula from
    \cref{c:oddordervalues}.
\end{proof}

\begin{thm}
\label{t:values}
    Assume \cref{basicsetup}, and that $R$ is a principal
    ideal ring.
    Let $g\in \Sp(V)$ and let 
    $B_g\colon V(1-g) \times V(1-g)\to R$ 
    be the form from \cref{l:sigmaform}.
    Then there exists a non-degenerate, symmetric form
    $q\colon V(1-g) \times V(1-g) \to R$.
    For any such form $q$, we have
    \begin{equation} 
    \label{equ:gencharvalue}
       \psi(g) = \sqrt{\abs{C_V(g)}} \:
                 \sign (q/B_g)\:
                  \gamma_{\lambda}(-q),
    \end{equation}
    where $\psi=\psi_{\omega, \lambda}$ 
    is the Weil character associated to $\omega, \lambda$.
\end{thm}

\begin{proof}
    We begin by noticing that the right hand sight of 
    Formula~\eqref{equ:gencharvalue}
    is independent of the choice of $q$:
    If $\widetilde{q}$ is another non-degenerate, symmetric form
    on $V(1-g)$,
    then 
    $\sign(\widetilde{q}/B_g)
      = \sign(q/B_g)\sign(\widetilde{q}/q)
    $,
    and so 
    $\sign(\widetilde{q}/B_g)\gamma_{\lambda}(-\widetilde{q})
     = \sign(q/B_g)\gamma_{\lambda}(-q)$
    by \cref{c:gauss_ae}.   

    We will show simultaneously that
    there is a non-degenerate symmetric form\footnote{%
        It is of course easy to show directly that there are 
        non-degenerate, symmetric, $R$-bilinear forms
        on $U=V(1-g)$
        (for example, from \cref{c:pirbilin}).}~$q$
    on $V(1-g)$ with
    $\sign(q/B_g)=1$ and 
    $\psi(g) = \sqrt{\abs{\C_V(g)}} \gamma_{\lambda}(-q)$.
    (By the first paragraph, this proves the theorem.)
    The proof will be by induction on 
    $\abs{V(1-g)}$.
    Assume that $g$ is a counterexample 
    with $\abs{V(1-g)}$ of minimal possible order,
    and write $U=V(1-g)$.

    First, assume that $U=V(1-g)$ is cyclic as $R$-module,
    that is, $U = R x$ for some $x\in U$.
    As $B_g( rx, sx ) = B_g( sx, rx )$ for all $r$, $s\in R$,
    the form $B_g$ itself is symmetric and non-degenerate.
    Thus we can choose $q = B_g = Q_g$.
    Then clearly $\sign(q/B_g)=1$.
    As 
    $  Rx \subseteq x^{\perp} 
    = ( V(1-g) )^{ \perp } =  \C_V(g)$,
    it follows that $(g-1)^2 = 0$ and thus
    the order of $g$ divides $\abs{R x}$.
    In particular, $g$ has odd order.
    Thus \cref{c:oddorderQg} yields
    $\psi(g) = \sqrt{\abs{\C_V(g)}} \gamma_{\lambda}( -Q_g )$,    
    and the theorem follows in this case.
    So in a counterexample, $U$ can not be a cyclic $R$-module.
    
   Next, suppose that $ U=V(1-g)$ has submodules $X$, $Y$
   such that 
   \begin{equation} 
     U = X \oplus Y
      \quad \text{with} \quad 
      B_g(Y,X) = 0
      \quad\text{and}\quad
      X\neq 0 \neq Y.
      \label{eq:uorthdecomp}
   \end{equation}
   Then by \cref{p:factorize}, we can write
   $g=hk$ with
   $V(1-h)=X$, $V(1-k)=Y$ and
   $B_h = (B_g)_{|X}$,
   $B_k= (B_g)_{|Y}$.
   By minimality of $U$, there are non-degenerate symmetric forms
   $q_h$ and $q_k$ on $X$ and $Y$
   with $\sign(q_h/B_h)=\sign(q_k/B_k)=1$
   and such that \eqref{equ:gencharvalue} holds for
   $h$, $k$.
   Set $q = q_h\oplus q_k$.
   By \cref{l:disc_dirprod},
   $\sign(q/B_g) = 1$.
   Then
   \begin{align*}
     \frac{\psi(g)}{\sqrt{\abs{V}}}
        &= \frac{\psi(h)}{\sqrt{\abs{V}}}
           \cdot 
           \frac{\psi(k)}{\sqrt{\abs{V}}}  
        && \text{(\cref{p:mult})} 
     \\ & = \frac{ \gamma_{\lambda}( -q_h )
                }{ \sqrt{ \abs{V(1-h)} } 
                }
            \cdot 
            \frac{ \gamma_{\lambda}( -q_k )
                }{ \sqrt{ \abs{V(1-k)} } 
                }
        && \text{(induction)}
    \\  &= \frac{ \gamma_{\lambda}( -q )
               }{ \sqrt{ \abs{V(1-g)} } 
               }           .
 && \text{(\cref{l:gauss_elem}~\ref{it:gaussmult})}  
   \end{align*}
   Thus $g$ is not a counterexample, contradiction.
   Thus there is no
   decomposition as in \eqref{eq:uorthdecomp}.
   
   Let $Q_g(x,y) = (1/2) \big( B_g(x,y) + B_g(y,x)
                         \big)$.
   By \cref{p:pirbilin_id},
   $I:=B_g(U,U)$ and $Q_g(U,U)$ are ideals of $R$.
   We claim that $Q_g(U,U) < I$.
   By \cref{c:pirsymform}, there is 
   $x\in U$ with
   $Q_g(U,U)=R Q_g(x,x) $.
   If $Q_g(U,U)=I$, then
   $B_g(U,U)=  I = R Q_g(x,x) = RB_g(x,x)$, 
   and
   \cref{p:pirbilin_id} yields that
   $U = R x \oplus \leftidx{^{B_g}}{x}$.
   By definition, $B_g(\leftidx{^{B_g}}{x},R_1x)=0$.
   When $\leftidx{^{B_g}}{x} \neq 0$, then
   we have a decomposition as in \eqref{eq:uorthdecomp},
   which contradicts the previous paragraph.
   When $\leftidx{^{B_g}}{x}=0$, 
   then $U=  Rx$ is cyclic and 
   $g$ is not a counterexample at all.
   Thus in a counterexample with $\abs{U}$ minimal,
   we must have $Q_g(U,U) < I$.
     (When $R$ is a field,
     then it follows at this point that $g$ is an involution
     (\cref{l:symformqg}),
     and the proof can be finished by an appeal to
     \cref{c:value_inv}, as in Ward's proof~\cite{Ward17}.
     We have to work a little bit harder here.)
      
   As $Q_g(U,U) < I = B_g(U,U)$,
   we have that $Q_g$ is degenerate.
   By \cref{l:symformqg}, $\Ker(g+1) \neq \{0\}$, and so
   $g$ has even order.
   Thus $\erz{g}$ contains a unique involution~$t$.
   Clearly, $\C_V(g) \leq \C_V(t)$ and thus 
   (by \cref{l:ker_senkr})
   $V(1-t) \leq V(1-g)=U$.
   We have the decomposition
   $V = \C_V(t) \oplus V(1-t)$
   which is orthogonal with respect to $\omega$.
   Intersecting with $U$ gives
   $U = (U\cap \C_V(t)) \oplus V(1-t)$.
   As $gt=tg$, we have $Ut=U$ and
   $B_g(xt,yt)=B_g(x,y)$ for $x$, $y\in U$.
   Therefore,
   $B_g(U\cap \C_V(t), V(1-t)) = 0$.
   By non-existence of a decomposition~\eqref{eq:uorthdecomp}, 
   we have
   $U \cap \C_V(t) = \{0\}$
   and thus
   $U = V(1-t)$ and $\C_V(t) = \C_V(g)$.
   
   Write $t=g^k$. Then $ut=-u$ for all $u\in U$.
   There is $u\neq 0$ in $U^{Q_g} = \Ker(g+1)$,
   that is, $ug=-u$. Then
   $-u = ut = ug^k = (-1)^k u$,
   so $k$ is odd.
   Thus $g= th $ with $h$ of odd order.

   We now apply \cref{p:mult} to $g=th$ and conclude
   \begin{align*}
     \psi(g) 
         &= \frac{\psi(t)\psi(h)}{\sqrt{\abs{V}}} 
            \sum_{u\in V(1-t)\cap V(1-h)}
         \lambda( \tfrac{1}{2}( B_t( u, u ) +B_h( u, u ) ) )
      \\ &= \frac{\psi(t)\psi(h)}{\sqrt{\abs{V}}} 
                  \sqrt{ \abs{ V(1-h) } } 
                        \gamma_{\lambda}( Q_h )
   \end{align*}
   as $V(1-h)\subseteq V(1-g)=V(1-t)=U$ and $B_t = 0$.
   By \cref{c:value_inv} and \cref{c:oddorderQg},
   it follows
   \begin{align*}
    \psi(g) 
         &= \frac{1}{\sqrt{\abs{V}}} 
            \, 
            \sqrt{ \abs{ \C_V(t) } } 
             (-1)^{ \frac{ \sqrt{\abs{U}} -1
                         }{2}
                  }
            \, \sqrt{ \abs{ \C_V(h) } }
                  \gamma_{\lambda }( - Q_h )
            \, \sqrt{ \abs{ V(1-h) } }  
                \gamma_{\lambda}( Q_h )
      \\ &= \sqrt{ \abs{ \C_V(t) } } 
           (-1)^{ \frac{ \sqrt{\abs{U}} -1
                       }{2}
                }
          = \sqrt{ \abs{ \C_V(g) } }
            (-1)^{ \frac{ \sqrt{\abs{U}} -1
                        }{2}
                  }
   \end{align*}
   To finish the proof, we have to show that
   there is a form~$q$ such that $\sign(q/B_g)=1$ and 
   $\gamma_{\lambda}(-q) = (-1)^{(\sqrt{\abs{U}}-1)/2}$.
   This will follow from the next lemma,
   which also contains the main work for the proof
   of \cref{ti:inv_value}:
   \renewcommand{\qedsymbol}{}
\end{proof}

\begin{lemma}
\label{l:semisimpleel}
    In the situation of \cref{t:values},
    assume that
    $\C_{V}(g) \cap U=0$, where $U = V(1-g)$.
    Then $(1-g)_{|U}$ is invertible and
    there is a symmetric form~$q$ such that
    \[  
       \sign (q/B_g) = \sign_U(1-g)
       \quad\text{and}\quad
    \gamma_{\lambda}(-q)  
        = 
        (-1)^{\frac{\sqrt{\abs{U}}-1}{2}}.
    \]
\end{lemma}
\begin{proof}
   As $\C_V(g)=\Ker(1-g)$,
   it is clear that $(1-g)_{|U}$ is invertible.
   As $\Ker(1-g) = U^{\perp}$,
   we have that $V = \Ker(1-g) \oplus U$ is an orthogonal
        sum with respect to the form $\omega$, 
        and thus $\omega\colon U\times U\to R$ is 
        non-degenerate.
        (In particular, $\abs{U}$ and $\abs{\C_V(g)}$ are squares.)

        By \cref{c:piralter}, we can write
        $U= Re_1 \oplus \dotsb \oplus Re_k \oplus  
        Rf_1 \oplus \dotsb \oplus Rf_k$,
        where $\omega(e_i,f_j) = \delta_{ij}d_i$ 
        and $\omega(e_i,e_j)=\omega(f_i,f_j)=0$ for all $ i$, $j$,
        and $Rd_1 \geq Rd_2 \geq \dotsb \geq Rd_k$.
        Notice that for $i$ fixed and $r\in R$, we have
        $re_i=0 \iff rf_i=0 \iff rd_i=0$.
        Thus we can define a non-degenerate, symmetric form
        $q$ on $U$
        by requiring 
        $q(e_i,e_j) = q(f_i,f_j)= \delta_{ij} d_i$ 
        and $q(e_i,f_j) = 0$ for all $i$, $j$.

        There is a unique automorphism $\alpha\in \Aut_R(U)$
       such that $e_i\alpha=f_i$, $f_i\alpha=-e_i$ for all $i$.
        For this $\alpha$, we have
        $\omega(x,y) = q(x\alpha,y)$.        
        Since $(1-g)_{|U}$ is invertible, we have for $x$, $y\in U$:
        \[ B_g(x,y) = \omega(x(1-g)^{-1},y)
           = q(x(1-g)^{-1}\alpha,y)
           \, .
        \]
        Thus 
        $\sign(q/B_g) = \sign_U((1-g)^{-1}\alpha)
         = \sign_U(1-g)\sign_U (\alpha)$.
         
        We claim that $\sign_U(\alpha)=1$.
        As $\alpha^2= -1$, 
        only $v=0$ is fixed by $\alpha^2$.
        Thus the cycle decomposition of $\alpha$ as permutation on $U$ consists of $(\abs{U}-1)/4$ cycles of length~$4$.
        As $\abs{U}$ is a square,  $(\abs{U}-1)/4$ is even and the claim
        follows. 
        It follows that $\sign(q/B_g)=\sign_U(1-g)$.
        
        To compute $\gamma_{\lambda}(-q)$, we 
        observe that 
        $(U,-q) \iso (L,-q_L) \oplus (L,-q_L)$,
        where 
        $L= Re_1\oplus \dotsb \oplus Re_k \iso 
            Rf_1 \oplus \dotsb \oplus Rf_k$
        and $q_L$ is the restriction of $q$ to $L$.
        It follows from
        \cref{l:gauss_elem}\ref{it:gaussmult} and
        \cref{p:gaussiansigns}\ref{it:gausssquare}
        that
        \[ \gamma_{\lambda}(-q) 
        =  \gamma_{\lambda}(-q_L)^2
        = (-1)^{\frac{\sqrt{\abs{U}} -1}{2}}.
        \]
        Now the result follows.        
\end{proof}

\begin{proof}[Proof of \cref{t:values}, continued]
    We are in the situation where
    $g=th$ with $t^2=1$ and $h$ of odd order,
    and $V(1-g)=V(1-t)=U$.    
   We claim that 
      $\sign_U(1-g)=1$ in this situation.
      On $U$, the element $t$ acts as $-1$.
      Thus 
      $0 < \Ker(1+g)=\Ker(1+th) \cap U = \Ker(1-h) \cap U$.
      It follows $V(1-h) < V(1-g)=U$.
      Thus by induction, \cref{t:values} holds for $h$.
      By comparison with \cref{c:oddorderQg},
      we must have $\sign_{V(1-h)}\left(\frac{1+h}{2}\right) = 1$.
      On $U/V(1-h)$, the element
      $\frac{1+h}{2} = \frac{h-1}{2} + 1$ acts as identity.
      So by \cref{l:signreduc},
      $\sign_U \left( \frac{1+h}{2}\right)= 1$.
      Together with $h_{|U} = -g_{|U}$, it follows
      $\sign_U \left(\frac{1-g}{2}\right) = 1$.
      As $\abs{U}$ is a square, we have
      $\sign_U\left(1/2\right)=1$ and thus
      $\sign_U(1-g)=1$ as claimed.
      
   Together with \cref{l:semisimpleel},
   it follows that $g$ is not a counterexample either.
   This is the final contradiction that finishes the proof
   of \cref{t:values}.       
\end{proof}

Now suppose that $R$ is as in \cref{basicsetup},
but not necessarily a PIR.
I do not know whether one can always
find a non-degenerate symmetric form 
$q\colon V(1-g) \times V(1-g) \to R$
in this case.
(Notice that we are given the non-degenerate,
but in general non-symmetric form $B_g$ on $V(1-g)$.)
Thus we assume the existence of $q$ in the next result,
which is \cref{ti:main} from the introduction:
\begin{cor}
\label{c:values_gen}
    Assume \cref{basicsetup}.
    Let $g\in \Sp(V)$ and let 
    $B_g\colon V(1-g) \times V(1-g)\to R$ be 
    the form from \cref{l:sigmaform}.
    If
    $q\colon V(1-g) \times V(1-g) \to R$ 
    is a non-degenerate, symmetric form,
    then
    \[ \psi(g) = \sqrt{\abs{C_V(g)}} \:
    \sign(q/B_g) \,\gamma_{\lambda}(-q).
    \]
\end{cor}
\begin{proof}
    Let $m$ be the order of $\lambda$.
    By \cref{rm:indepR},
    we can replace the data $(V,R,\omega,\lambda)$   
    by the data $(V,R'=\ints/m\ints,\kappa\circ\omega,\lambda')$,
    where
    $\kappa \colon R \to R'$
    and 
    $\lambda'\colon R'\to \compl^*$ 
    are such that $\lambda = \lambda' \circ \kappa$,
    without changing $\psi(g)$.
    The form $\kappa\circ B_g$ is the form belonging to 
    $g$ with respect to $\kappa\circ \omega$.
    When $q(x,y)= B_g(x\alpha,y)$ for $\alpha\in \Aut_R(V)$,
    then also 
    $\kappa(q(x,y))=\kappa(B_g(x\alpha,y))$
    and thus 
    $\sign(\kappa\circ q/\kappa\circ B_g)
    = \sign(q/B_g)$.
    Clearly,
    $\gamma_{\lambda}(-q) 
     = \gamma_{\lambda'}(-(\kappa\circ q))$.
    As $R'$ is a principal ideal ring, the result follows
    from \cref{t:values}.
\end{proof}

When we can not find a symmetric, non-degenerate form 
$q\colon V(1-g)\times V(1-g)\to R$,
then we can always replace $R$ by $R'=\ints/m\ints$
as in the above proof,
and find a form $q$ with values in $R'$,
so  that we can evaluate the formula from
\cref{t:values}.

Finally, the next corollary contains
\cref{ti:inv_value}, 
which is the case $\C_V(g)=\Ker(1-g)=\{0\}$.

\begin{cor}
    \label{c:gm1_inv}
    Let $g\in \Sp(V)$ and set $U=V(1-g)$. 
    Assume that
    $\C_{V}(g) \cap U=0$.    
    Then $(1-g)_{|U}$ is invertible and
    \[ \psi(g) = \sqrt{ \abs{\C_V(g)} } 
    (-1)^{\frac{\sqrt{\abs{U}} -1}{2}}
    \sign_U(1-g).
    \]    
\end{cor}

When $R$ is a field, or more generally,
when $U\iso R^{2k}$, then
$\sign_U(1-g) = \sign_R(\det(1-g)_{|U})$ by \cref{l:detsign}. 
Thus this corollary generalizes a result of
Gurevich and Hadani~\cite{gurevichhadani07}.

\begin{proof}[Proof of \cref{c:gm1_inv}]
    Without loss of generality, we may assume that
    $R= \ints/m\ints$ for some odd integer $m$.
    The result follows then from 
    \cref{t:values} and \cref{l:semisimpleel}.
\end{proof}

\section{Corollaries and Examples}
\label{sec:corexpl}

We assume \cref{basicsetup}.
As before, $\psi$ denotes the character of the canonical Weil 
representation.

\begin{example}
\label{ex:field}
    Let $R=\GF{q}$ be a finite field.
    Then the symmetric forms $q_g$ on $V(1-g)$ with 
    $\sign(q_g/B_g)=1$ are exactly the symmetric forms on $V(1-g)$
    with the same discriminant
    as $B_g$.
    We get the formulas
    \[ \psi(g) = \sqrt{ \abs{ \C_V(g) } } 
                   \gamma_{\lambda}(-q_g)
               = \sqrt{ \abs{ \C_V(g) } } 
                  \,
                   \gamma_{\lambda}(-1)^{\dim V(1-g)} 
                   \sign_{\GF{q}}(\disc B_g).
    \]
(These are essentially the formulas obtained by
T.~Thomas~\cite{thomas08,Thomas13}.
Recall that $B_g=-\sigma_g$,
with $\sigma_g$ as in \cite{thomas08,Thomas13}.)
\end{example}

\begin{proof}
   By \cref{l:signdisc},
   $\sign(q_g/B_g)= 1$
    if and only if 
    $ \disc q_g = \disc B_g $.
    
    Let $\disc B_g = d (\GF{q})^2$, say.
    A possible choice for
    $q_g$ is the diagonal form
    $\langle 1, 1, \dotsc, 1, d \rangle$.
    Then 
    \[ \gamma_{\lambda}(-q_g)
       = \gamma_{\lambda}(-1)^{\dim V(1-g) -1} \gamma_{\lambda}(-d) 
       = \gamma_{\lambda}(-1)^{\dim V(1-g)} \sign_{\GF{q}}(d)
    \]
    by \cref{l:gauss_elem,c:gauss_ae}.    
\end{proof}

The following observation, already used in the last step
of the proof of \cref{t:values},
seems a little bit curious:
\begin{cor}
    Suppose $g\in \Sp(V)$ has odd order.
    Then $\sign_V (1+g) = 1$.
\end{cor}
\begin{proof}
   By \cref{p:oddupsi}, 
   $\psi(-g)= \psi(-1) = (-1)^{ \frac{ \sqrt{\abs{V}} - 1 
                                    }{ 2 }
                              }$.
   By \cref{c:gm1_inv} or \cref{ti:inv_value}
   applied to $-g$,
   we have
   $\psi(-g) = (-1)^{ \frac{ \sqrt{\abs{V}} - 1 
                                       }{ 2 }
                                 }
               \sign_V(1+g)$.
   Thus the result.
\end{proof}

When $V(1-g) \subseteq \C_V(g)$,
then $\sign_V(1+g)=1$ is trivial.
But I do not see how to prove this directly
in more general situations.

When $g$ has even order, then
$1+g$ may not be invertible.
Even when $\Ker(1+g)=\{0\}$,
then in general
$\sign_V(1+g)\neq 1$.
An example is
$g = \begin{psmallmatrix}
        0 & 1 \\ 
        -1 & 0
\end{psmallmatrix}
\in \SL(2,\ints/m\ints) = \Sp( 2,\ints/m\ints)$,
where $\omega$ is the standard symplectic form on $(\ints/m\ints)^2$
(with Gram matrix again $g$), and $m \equiv \pm 3 \mod 8$.
Then $\sign_V(1+g) = \Jacobi{2}{m} = -1$.
The same $g$ also shows that
the formula from \cref{c:oddordervalues} does not hold for 
$g$ of even order, even when $\Ker(1+g)=\{0\}$.

\begin{remark}
\label{rm:orthsumweilchar}
    Suppose that $V= V_1 \oplus V_2$ with
    $V_1^{\perp}=V_2$. 
        Then $\omega$ is non-degenerate on $V_1$ and $V_2$.
    Let $\psi_i$ be the Weil character of $\Sp(V_i)$
    and $g_i\in \Sp(V_i)$.
    Then $g=(g_1,g_2)\in \Sp(V)$
    via $(v_1 + v_2)(g_1,g_2)=v_1g_1+v_2g_2$,
    and $\psi(g) = \psi_1(g_{1}) \psi_2( g_{2})$. 
\end{remark}
\begin{proof}
    We have $V(1-g) = V_1(1-g_1) \oplus V_2(1-g_2)$
    and $B_g = B_{g_1} \oplus B_{g_2}$,
    and thus from the formula in \cref{t:values},
    we see $\psi(g) = \psi_1(g_{1}) \psi_2( g_{2})$.
\end{proof}

Of course, this is well known 
and with just a little bit more effort, 
we could have proved
this remark in \cref{sec:weilrep}.
It is also well known that studying the Weil representation
over a finite ring can be reduced to 
studying the Weil representation over 
a finite, \emph{local} ring, as follows:

\begin{cor}
\label{c:ringdirprod_char}
   The finite ring $R$ can be written as
   the direct product of finite local rings,
   say $R=R_1 \times \dotsm \times R_{\ell}$.
   Then there is an orthogonal decomposition
   $V=V_1\oplus \dotsb \oplus V_{\ell}$,
   where each $V_i$ is a module over $R_i$,
   and $ \Sp_R(V) \iso \Sp_{R_1}(V_1) \times \dotsm \times 
                   \Sp_{R_{\ell}}(V_{\ell}) $.
   Moreover,
   $\psi = \psi_1 \times \dotsm \times \psi_{\ell}$,
   where $\psi_i$ is the Weil character of type 
   $\lambda_i:= \lambda_{|R_i}$ associated to 
   $V_i$, $\omega\colon V_i\times V_i\to R_i$.
\end{cor}
\begin{proof}
    This is fairly standard.
    The product decomposition of $R$ is a standard result from
    commutative ring theory
    \cite[Theorem~8.7]{AtiyahMacdonald69}.
   Let $e_i$ be the identity of $R_i$.
   Then $e_ie_j=\delta_{ij}e_i$ and
   $1_R=e_1+\dotsb + e_{\ell}$.
   Set $V_i=e_iV$.
   Then $\omega(V_i,V_j)\subseteq e_ie_jR = \delta_{ij} R_i$,
   so that the above decomposition of $V$ is orthogonal.
   Each $V_i$ is invariant under $\Sp_R(V)$.
   Thus 
   \[ \Sp_R(V)\ni g \mapsto g_{|V_1} \times \dotsm \times 
                           g_{|V_{\ell}} 
             \in \Sp_{R_1}(V_1) \times \dotsm \times 
                                \Sp_{R_{\ell}}(V_{\ell})
   \]
   defines an isomorphism.
   The claim on the Weil character follows from 
   \cref{rm:orthsumweilchar}.
\end{proof}

\begin{cor}
    For $g\in \Sp(V)$, 
    the character value $\psi(g)$ is rational
    if and only if 
    $\abs{V(1-g)}$ is a square.
\end{cor}
\begin{proof}
    $\abs{V(1-g)}$ is a square
    if and only if
    $\abs{\C_V(g)} = \psi(1)^2/ \abs{ V(1-g) }$
    is a square.
    \enquote{Only if} follows already from
    \cref{p:absvalue}.
    Conversely, when $\abs{V(1-g)}$ is a square,
    then by \cref{p:gaussiansigns}~\ref{it:gausssquare},
    $\gamma_{\lambda}(-q)= \pm 1$ for any 
    non-degenerate, symmetric form $q$ on $V(1-g)$.
    Thus $\psi(g)\in \rats$.
\end{proof}

When $(V,R,\omega',\lambda')$ is another data satisfying
\cref{basicsetup}, we have Weil characters
$\psi_{\omega, \lambda}$ and
$\psi_{\omega', \lambda'}$ associated to this data.

\begin{prop}[Changing $\omega$]
\label{p:changeomega}
   Suppose that $\omega'\colon V\times V\to R$ is another
   non-degenerate, alternating form,
   and that $g\in \Sp_R(V,\omega) \cap \Sp_R(V,\omega')$.
   Then there is $a\in \GL_R(V)$ such that
   $\omega'(v,w) = \omega(va,w)$ for all $v$, $w\in V$.
   For this $a$, we have $ag=ga$
   and 
   \[ \psi_{\omega', \lambda}(g) 
      = \sign_{V(1-g)}(a) \: \psi_{\omega, \lambda}(g).
   \]
\end{prop}
\begin{proof}
   Since $\omega'$ and $\omega$ both induce
   isomorphisms $V\to \Hom_R(V,R)$,
   there is some $a\in \GL_R(V)$ such that
   $\omega'(v,w) = \omega(va,w)$ for all $v$, $w\in V$.
   From $g\in \Sp_R(V,\omega) \cap \Sp_R(V,\omega')$, it follows
   \begin{align*}
     \omega'(v,w) &= \omega'(vg,wg) = \omega(vga,wg) \quad\text{and}\\
     \omega'(v,w) &= \omega(va,w) = \omega(vag,wg)
   \end{align*}
   for all $v$, $w\in V$, and thus $vag=vga$.
   In particular, $V(1-g)$ is $a$-invariant and $\sign_{V(1-g)}(a)$
   is defined.
   
   Let $B_g$ and $B_g'$ be the forms associated to 
   $\omega$ and $\omega'$, respectively.
   Then $B_g'(x,y) = B_g(xa,y)$.
   Without loss of generality, we can assume that there is a
   non-degenerate, symmetric form $q$ on $V(1-g)$
   (if necessary, replace $R$ by $\ints/m\ints$).
   Then $\sign(q/B_g')= \sign(q/B_g) \sign_{V(1-g)}(a)$,
   and the result follows from \cref{t:values}.
\end{proof}

For $\lambda\colon R\to \compl^*$ an additive character and $s\in R$,
define $\lambda s\colon R\to \compl^*$ by $(\lambda s)(r)=\lambda(sr)$.
When $\lambda$ is primitive, then $\lambda s = 1 \iff s= 0$
and thus every character of $R$ has the form $\lambda s$.
The character $\lambda s$ is again primitive if and only if
$s$ is a unit in $R$
\cite[\S~3]{Honold01}.

\begin{cor}[Changing $\lambda$]
\label{c:changelambda}
    Let $\lambda' =\lambda s \colon R\to \compl^*$ 
    be another primitive character of $R$.
    Then 
    \[\psi_{\omega, \lambda'}(g) 
         = \sign_{V(1-g)} (s)\;  
           \psi_{\omega, \lambda}(g)
       \quad \text{for all} \quad
       g\in \Sp_R(V,\omega),
    \]
    where $\sign_{V(1-g)}(s)$ 
    is the sign of the permutation on $V(1-g)$
    induced by multiplication with the unit~$s$.
\end{cor}
\begin{proof}
    We have $(\lambda s) \circ \omega = \lambda \circ (s\omega)$,
    where $(s\omega)(v,w)=s\omega(v,w)=\omega(vs,w)$,
    and $\Sp_R(V,\omega)=\Sp_R(V,s\omega)$,
    so this follows from \cref{p:changeomega}.
\end{proof}

    Notice that 
    $\psi_{\omega, \lambda}$ 
    and 
    $\psi_{\omega, \lambda'}$ 
    are not necessarily algebraic conjugates
    (for example, when $R$ is a field of square order,
    or a local ring where the residue field has square
    order).        
    In general, \cref{c:changelambda} does not hold for 
    $g\in \Sp_{\ints}(V,\lambda\circ\omega)$,
    because in general,
    $\Sp_{\ints}(V,\lambda\circ\omega) 
      \neq \Sp_{\ints}(V,\lambda'\circ\omega)$.
      (These groups are isomorphic
      and in fact conjugate in $\Aut_{\ints}(V)$.)
    On the other hand, when
    $\Ker \lambda = \Ker \lambda'$,
    then     
    $\Sp_{\ints}(V,\lambda\circ\omega) 
     = \Sp_{\ints}(V,\lambda'\circ\omega)$    
    and $\lambda' = \lambda s = \lambda^s$ for
    $s\in (\ints/m\ints)^*$,
    and the corresponding Weil characters are algebraic conjugates.

    In the following, we write 
    $\psi_{\lambda}:= \psi_{\omega, \lambda}$,
    as $\omega$ will be fixed.
\begin{cor}
\label{c:indep_cond}
    Assume \cref{basicsetup}.  
    \begin{enumerate}
    \item 
    \label{it:fixunit_s}%
        Let $s\in R^*$.
        Then $\psi_{\lambda}=\psi_{\lambda s}$
        as characters on $\Sp_R(V,\omega)$ 
        if and only if $s$ is a square
        in $R/\ann_R(V)$.
    \item 
    \label{it:fixgrpel_g}%
        Let $g\in \Sp_R(V,\omega)$. 
        Then $\psi_{\lambda}(g) = \psi_{\lambda s}(g)$
        for all $s\in R^*$
        if and only if every simple $R$-module
        occurs with even multiplicity 
        in any composition series of $V(1-g)$.
        (In the case of a local ring $R$,
        this means that $V(1-g)$ has even length.)
    \end{enumerate}
\end{cor}
\begin{proof}
    In view of \cref{c:ringdirprod_char}, both assertions
    reduce easily to the case where $R$ is local,
    so we assume this.
    Let $J$ be the unique maximal ideal of $R$.
    
    When $s$ is a square in $R/\ann_R(V)$, then
    $\sign_X(s)=1$ for any $R$-submodule~$X$ of $V$
    and thus $\psi_{\lambda} = \psi_{\lambda s}$.
    
    Now assume that $s$ is not a square in $R/\ann_R(V)$.
    The unit group of the finite local ring $R$ has the structure
    $R^* = (R/J)^* \times (1+ J)$.
    As $\abs{R}$ is odd, also $\abs{1+J}=\abs{J}$ is odd and thus every 
    element in $1+J$ is a square in $1+J$.
    Therefore, $s$ is not a square in the field $R/J$.
    Let $0<U\leq V$ be a minimal submodule.
    Then $U\iso R/J$, and we have $\sign_{U}(s)=-1$.
    It follows from \cref{p:spparam}
    that there is an element $g\in \Sp(V)$ with $V(1-g)=U$.
    (We have $U = Ru$ and there are nonzero
    symmetric forms $U\times U \to I\subseteq R$,
    where $I=\ann_{R}(J)$.)
    For such $g$, we have
     $\psi_{\lambda}(g) = - \psi_{\lambda s}(g)$.     
    This shows \ref{it:fixunit_s}.

    Now fix $g\in \Sp(V)$,
    and let $0= U_0 < U_1 < \dotsb < U_{\ell}=V(1-g)$ 
    be a composition series of $V(1-g)$.
    Every composition factor is isomorphic to the unique
    simple $R$-module $R/J$.
    Thus
    $\sign_{V(1-g)}(s) = [\sign_{R/J}(s)]^{\ell}$ for any 
    $s\in R^*$ by \cref{l:signreduc}.
    When $s\in R^*$ is such that $s+J$ is not a square in $R/J$,
    then $\sign_{R/J}(s)=-1$.
    Thus \ref{it:fixgrpel_g} follows from \cref{c:changelambda}.
\end{proof}

As another application,
we reprove
a result of R.~Guralnick, K.~Magaard and P.~H.~Tiep 
\cite[Theorem~1.2]{GuralnickMagaardTiep18},
and extend it
from finite fields to finite rings. 
We begin with a lemma:

\begin{lemma}
\label{l:minusgsquare}
   Let $g\in \Sp_R(V,\omega)$.
   Then $\psi_{\lambda}(-g^2)$ is independent of the primitive
   character $\lambda\colon R\to\compl$.
\end{lemma}
\begin{proof}
   By \cref{c:ringdirprod_char}, we may assume that $R$ is local,
   with maximal ideal $J$.
   By \cref{c:indep_cond}, we need to show that
   $V(1+g^2)$ has even length.
   Equivalently, we can show that $\Ker(1+g^2)$ has even length.
   Set $q:= \abs{R/J}$ and let $\ell$ be the length of 
   $\Ker(1+g^2)$.
   Then $\abs{\Ker(1+g^2)} = q^{\ell}$.
   For $v\in \Ker(1+g^2)$, we have $vg^2 = -v$.
   Thus the $\erz{g}$-orbit of any $v\in \Ker(1+g^2)$ has $4$ elements,
   except for $v= 0$.
   Thus $q^{\ell} \equiv 1 \mod 4$.
   It follows that when $q\equiv 3\mod 4$, then $\ell$ is even
   as claimed.
   Suppose $q\equiv 1 \mod 4$.
   Then $-1$ is a square in $R/J$, and thus also in $R$,
   say $-1 = i^2$ with $i\in R$.
   We have $1+g^2 = (g+i)(g-i)$.
   For $v\in \Ker(1+g^2)$, we can write
   \[ v = \frac{1}{2} v(1+ig) + \frac{1}{2}v(1-ig)
        \in \Ker(g+i) + \Ker(g-i).
   \]
   Thus
   \[ \Ker(1+g^2) = \Ker(g+i) \oplus \Ker(g-i).
   \]
   By \cref{l:ker_senkr},
   $\Ker(g + i) = V(g - i)^{\perp}$.
   Thus
   \[ \abs{\Ker(g+i)} 
      = \abs{V(g-i)^{\perp}}
      = \abs{V}/ \abs{V(g-i)} 
      = \abs{\Ker(g-i)},
   \]
   and thus $\Ker(1+g^2)$ has even length as claimed.
\end{proof}

\begin{cor}
\label{c:GMT18}
   For $g\in \Sp_R(V)$ and $\lambda\colon R\to \compl^*$ primitive,
   we have
   \[ \psi_{\lambda}(g) \psi_{\lambda}(-g) 
      = (-1)^{ \frac{ \sqrt{\abs{V}} -1 }{ 2 } } 
        \psi_{\lambda^2}(g^2).
   \]
\end{cor}

Notice that $\psi_{\lambda}=\psi_{\lambda^2}$
on $\Sp_R(V,\omega)$ if and only if $2$ is a square in $R$.
When $R$ is local with residue field $R/J$ of order $q$,
then $2$ is a square in $R$ if and only if $q\equiv \pm 1 \mod 8$.
In particular, \cref{c:GMT18} for a finite field is a result by
Guralnick, Magaard and Tiep
\cite[Theorem~1.2]{GuralnickMagaardTiep18}.

\begin{proof}[Proof of \cref{c:GMT18}]
   We want to apply the convolution formula from
   \cref{p:mult} to $-g^2 = g\cdot (-g) $.
   From $v= (1/2)\big( v(1+g) + v(1-g) \big)$ we see that
   $\Ker(1-g^2) = \Ker(1-g)\oplus \Ker(1+g)$.
   It follows that
   \begin{align*}
     V(1-g^2) = \Ker(1-g^2)^{\perp} 
              &= \Ker(1-g)^{\perp} \cap \Ker(1+g)^{\perp}
           \\ &= V(1-g) \cap V(1+g).
   \end{align*}
   Next, for $x=v(1-g^2)\in V(1-g) \cap V(1+g)$, we have
   \begin{align*}
     B_g(x,x) + B_{-g}(x,x)
     &= \omega( v(1+g), v(1-g^2) ) + \omega( v(1-g), v(1-g^2) )
     \\ &= 2\omega(v,v(1-g^2))
         = 2 B_{g^2}(x,x).
   \end{align*}
   The convolution formula yields
   \[ \psi_{\lambda}(-g^2) 
      = \frac{\psi_{\lambda}(g)\psi_{\lambda}(-g)}{\sqrt{\abs{V}}}
         \sum_{x\in V(1-g^2)} \lambda(B_{g^2}(x,x)).
   \]
   The convolution formula applied to $\psi_{\lambda^2}$
   and $-g^2 = -1\cdot g^2$ yields
   \begin{align*}
     \psi_{\lambda^2}(-g^2)
       &= \frac{ \psi_{\lambda^2}(-1) \psi_{\lambda^2}(g^2) 
              }{ \sqrt{\abs{V}} }
          \sum_{ x\in V(1-(-1))\cap V(1-g^2) }
               \lambda^2( \frac{1}{2} 
                          ( \underbrace{B_{-1}(x,x)}_{=0} 
                           + B_{g^2}(x,x)
                          )
                        )
       \\
       &=  (-1)^{ (\sqrt{ \abs{V} } -1 )/2 } \;
           \frac{ \psi_{\lambda^2}(g^2) }{ \sqrt{\abs{V}} } 
           \sum_{x\in V(1-g^2)} \lambda(B_{g^2}(x,x)).
   \end{align*}
   But by \cref{l:minusgsquare},
   $\psi_{\lambda}(-g^2)=\psi_{\lambda^2}(-g^2)$.
   Cancelling the common factors from the two expressions,
   we get the result.
\end{proof}

\printbibliography[heading=bibintoc]

\end{document}